\documentclass[a4paper,10pt,openany]{amsart}
\usepackage{amssymb,latexsym,amscd, amsthm, amsmath}
\usepackage{graphicx,color, tikz} 
\newtheorem{Thm}{Theorem}[section]
\newtheorem{Cor}[Thm]{Corollary}
\newtheorem{Prop}[Thm]{Proposition}
\newtheorem{Lem}[Thm]{Lemma}
\newtheorem{Def}[Thm]{Definition}
\newtheorem{Ex}[Thm]{Example}
\newtheorem{Rmk}[Thm]{Remark}

\date{}

\newcommand{\AAA}{{\mathcal{A}}}

\newcommand{\CC}{{\mathbb{C}}}
\newcommand{\PP}{{\mathbb{P}}}

\newcommand{\ZZ}{{\mathbb{Z}}}
\newcommand{\NN}{{\mathbb{N}}}

\newcommand{\hada} {\star}

\begin{document}

\title{Gorenstein points in $\PP^3$ via Hadamard products of projective varieties}
\author{C. Bocci}
\address{Dipartimento di Ingegneria dell'Informazione e Scienze Matematiche,
Universit\`a di Siena\\
Via Roma, 56 Siena, Italy}
\email{cristiano.bocci@unisi.it}
\author{C. Capresi}
\address{Dipartimento di Ingegneria dell'Informazione e Scienze Matematiche,
Universit\`a di Siena\\
Via Roma, 56 Siena, Italy}
\email{capresi3@student.unisi.it}
\author{D. Carrucoli}
\address{Dipartimento di Ingegneria dell'Informazione e Scienze Matematiche,
Universit\`a di Siena\\
Via Roma, 56 Siena, Italy}
\email{danielecarrucoli@hotmail.it}

\begin{abstract} We show how to construct a stick figure of lines in $\PP^3$ using the Hadamard product of projective varieties. Then, applying the results of Migliore and Nagel, we use such a stick figure to build a Gorenstein set of points with given $h-$vector ${\mathbf h}$. Since the Hadamard product is a coordinate-wise product, we show, at the end, how the coordinates of the points, in the Gorenstein set, can be directly determined.
\end{abstract}

\maketitle

\section{Introduction}

In the last few years, the Hadamard products of projective varieties have been widely studied from the point of view of Projective Geometry and Tropical Geometry.
In fact, the Hadamard products of projective varieties and the Hadamard powers of a projective variety are  well-connected to other operations of varieties: they are the multiplicative analogs of joins and secant varieties, and in tropical geometry, tropicalized Hadamard products equal Minkowski sums.  It is natural to study properties of this new operation, and see its effects on various varieties.  

From the point of view of Projective Geometry, several directions of research have been considered.
The paper \cite{BCK}, where Hadamard product of general linear spaces is studied,  can be considered the first step in this direction. Successively, the first author, with Calussi, Fatabbi and Lorenzini, in \cite{BCFL1}, address the Hadamard product of linear varieties not necessarily in general position, obtaining, in $\mathbb{P}^2$ a complete description of the possible outcomes. 
Then, in \cite{BCFL2}, they address the Hadamard product of  not necessarily generic linear varieties and show that the Hilbert function of the Hadamard product $X\star Y$ of two  varieties, with $\dim(X), \dim(Y)\leq 1$, is the product of the Hilbert functions of the original varieties $X$ and $Y$ and  that the  Hadamard
product of two generic linear varieties $X$ and $Y$ is projectively equivalent to a Segre embedding.
 An important result contained in \cite{BCK} concerns the construction of star configurations of points, via Hadamard product. This result found a generalization in \cite{CCGVT} where the authors introduce a new construction, using the Hadamard product, to obtain star configurations of codimension $c$ of $\PP^n$ and which they called Hadamard star configurations. Successively,  Bahmani Jafarloo and Calussi, in  \cite{BJC},  introduce a more general type of Hadamard star configuration; any star configuration constructed by their approach is called a weak Hadamard star configuration.

The use of Hadamard products in this context permits a complete control both in the coordinates of the points forming the star configuration and  the equations of the hyperplanes involved on it. Thus, the question if other interesting geometrical objects can be obtained by Hadamard products naturally arises. In this paper, we give a first positive answer showing how to construct a Gorenstein set of points in $\PP^3$ with given $h-$vector, via Hadamard products.

Our approach is related to the well-known construction of Migliore and Nagel \cite{MN}, based on Liasion Theory, where the Gorenstein set of points is obtained as the intersection of two aCM curves, linked by a complete intersection  which is a stick figure of lines. We want to point out, one more time, that our method permits a complete control of the coordinates of the points in the Gorenstein set, and, moreover, this allows one to build such set in an easy algorithmic way. Briefly speaking, we use  suitable values $\AAA=\{\alpha_i,\beta_i\}$, $i=0,\dots, 3$ to define a line $L^\AAA$ and two sets of collinear points $\{P^\AAA_i\}$ and  $\{Q^\AAA_j\}$. In Theorem \ref{IC} we prove that the set $Z_{a,b}^\AAA$, consisting of the Hadamard products $P_i^\AAA\hada Q_j^\AAA$, for a suitable choice of indices $i$ and $j$, is a planar complete intersection. As pointed out in Remark \ref{BCFLrmk}, it is not true, in general, that the Hadamard product of two sets of collinear points, in $\PP^3$, gives rise to a planar complete intersection. As a matter of fact, the result follows from an ad hoc choice of the points  $P^\AAA_i$ and  $Q^\AAA_j$. Successively, we compute the Hadamard product $Z_{a.b}^\AAA\star L^\AAA$ and, in Theorem \ref{Thmstick}, we prove that $Z_{a.b}^\AAA\star L^\AAA$ is a stick figure of lines, which is exactly the required one for the construction of the Gorenstein set of points in \cite{MN}.
 
The paper is organized in the following way.

In Section \ref{basic} we recall  the definitions of a Hadamard product of varieties and Hadamard powers. We recall some results about Hadamard transformations, contained in \cite{BC}, leading to Theorem \ref{genset}, which proves the connection between the ideals of $V$ and $P\star V$, where $V\subset \PP^n$ is a variety and $P\in \PP^n$ is a point without zero coordinates.

In Section \ref{MNsec} we recall the construction of a Gorenstein set  in $\PP^3$ from the $h-$vector, as introduced in \cite{MN}.

In Sections \ref{planarsec} and \ref{sticksec} we define the objects $L^\AAA$, $P^\AAA_i$ and  $Q^\AAA_j$ involved in our construction. We also show some preliminary results about these objects. These results lead to Theorem \ref{IC} stating that $Z_{a,b}^\AAA$ is a planar complete intersections and to Theorem \ref{Thmstick}, stating that $Z_{a,b}^\AAA\star L^\AAA$ is a stick figure of lines.

Finally, in Section \ref{gorsec} we describe the Gorenstein set of points obtained from $Z_{a,b}^\AAA\star L^\AAA$ by the method of Migliore and Nagel.

During the whole paper, we work over the complex field $\CC$.

We wish to thank the referee for his/her very accurate reading of the paper and for his/her helpful suggestions.

\section{Basic facts on Hadamard product of varieties}\label{basic}

The Hadamard product of points in a projective space is a coordinate-wise product as in the case of the Hadamard product of matrices.

\begin{Def} Let $p,q\in \PP^n$ be two points with coordinates $[p_0:p_1:\cdots:p_n]$ and $[q_0:q_1:\cdots:q_n]$ respectively. If $p_iq_i\not= 0$ for some $i$, the Hadamard product $p\hada q$ of $p$ and $q$, is defined as
\[
p\hada q=[p_0q_0:p_1q_1:\cdots:p_nq_n].
\]
If  $p_iq_i= 0$ for all $i=0,\dots, n$ then we say $p\hada q$ is not defined.
\end{Def}

This definition extends to the Hadamard product of varieties in the following way.
\begin{Def}
Let $X$ and $Y$ be two varieties in $\PP^n$. Then the Hadamard product $X\hada Y$ is defined as

\[X\hada Y=\overline{\{p\hada q: p\in X, q\in Y, p\hada q\mbox{ is defined} \}}.\]
\end{Def}

\begin{Rmk}\rm The Hadamard product $X\hada Y$ can be  given in terms of composition of the Segre product and projection. Consider the usual Segre product
\[X\times Y\subset\PP^N\]
\[([\alpha_0:\dots : \alpha_n], [\beta_0: \dots :\beta_n])\mapsto [\alpha_0\beta_0 :\alpha_0\beta_1:\cdots :\alpha_n\beta_n]\]
and denote with $z_{ij}$ the coordinates in $\PP^N$. Let $\pi:\PP^N \dashrightarrow \PP^n$ be the projection map from the linear space $\Lambda$ defined by equations $z_{ii}=0,i=0,\ldots,n$. The {\em Hadamard product} of $X$ and $Y$ is
\[X\hada Y=\overline{\pi(X\times Y)},\]
where the closure is taken in the Zariski topology.
\end{Rmk}

\begin{Rmk}\rm
Let  $\mathbb{K}[\mathbf{x}]=\mathbb{K}\left[x_{0}, \ldots, x_{n}\right]$  be a polynomial ring over an algebraically closed field.

Let $I_1, I_2, \dots I_r$ be ideals in $\mathbb{K}[\mathbf{x}]$. We introduce $(n+1)r$ variables, grouped in $r$ vectors $\mathbf{y}_j=(y_{j0},\dots, y_{jn})$, $j=1,2,\dots, r$ and we consider the polynomial ring $\mathbb{K}[\mathbf{x},\mathbf{y}]$ in all $(n+1)(r+1)$ variables.

Let $I_j(\mathbf{y}_j)$ be the image of the ideal $I_j$ in $\mathbb{K}[\mathbf{x},\mathbf{y}]$ under the map $\mathbf{x} \mapsto \mathbf{y}$. Then the  Hadamard product $I_1\star I_2\star\cdots \star I_r$ is the elimination ideal
\[
\left(I_1(\mathbf{y_1})+\cdots + I_r(\mathbf{y_r})+\left\langle x_{i}-y_{1i}y_{2i}\cdots y_{ri} \mid i=0, \ldots, n\right\rangle\right) \cap \mathbb{K}[\mathbf{x}].
\]
The defining ideal of the Hadamard product $X\star Y$ of two varieties $X$ and $Y$, that is, the ideal $I(X\star Y)$, equals  the Hadamard product  of the ideals $I(X)\star I(Y)$  \cite[Remark 2.6]{BCK}.
\end{Rmk}

As in \cite{BCK} we give the following definition.

\begin{Def}
Let $H_i\subset\PP^n,i=0,\ldots,n$, be the hyperplane $x_i=0$ and set
$$\Delta_i=\bigcup_{0\leq j_1<\ldots<j_{n-i}\leq n}H_{j_1}\cap\ldots\cap H_{j_{n-i}}.$$
\end{Def}

In other words, $\Delta_i$ is the $i-$dimensional variety of points having {\em at most} $i+1$ non-zero coordinates. Thus $\Delta_0$ is the set of coordinates points and $\Delta_{n-1}$ is the union of the coordinate hyperplanes. Note that elements of $\Delta_i$ have {\em at least} $n-i$ zero coordinates. We have the following chain of inclusions:
\begin{equation}\label{inclusiondelta}
\Delta_0=\{[1:0:\cdots:0],\cdots,[0:\cdots:0:1]\}\subset\Delta_1\subset\ldots\subset\Delta_{n-1}\subset\Delta_{n}=\PP^n.\end{equation}

We end this section recalling some useful results contained in \cite{BC} and \cite{BCK}.

\begin{Lem}\label{LempL}
Let $L\subset\PP^n$ be a linear space of dimension $m$. Then, for a point $P\in\PP^n$,  $P\hada L$ is either empty or it is a linear space of dimension at most $m$. If $P\not\in\Delta_{n-1}$, then $\dim( P\hada L)=m$.
\end{Lem}

\begin{Lem}\label{LempqL}
Let $L\subset\PP^n$ be a linear space of dimension $m<n$ and consider points $P,Q\in\PP^n\setminus\Delta_{n-1}$. If $P\neq Q$, $L\cap\Delta_{n-m-1}=\emptyset$, and $\langle P,Q\rangle\cap\Delta_{n-m-2}=\emptyset$, then $P\hada L\neq Q\hada L$.
\end{Lem}

\begin{Lem}\label{Lempqr}
Let $P,Q_1,Q_2$ be three points in $\PP^n$ with $P\notin \Delta_{n-1}$. Then $P \hada Q_1 = P \hada Q_2$ if and only if $Q_1=Q_2$.
\end{Lem}

If $I=(i_0, \dots, i_n)$ is a vector of nonnegative integers, we denote by $X^I$ the monomial $x_0^{i_0}x_1^{i_1}\cdots x_n^{i_n}$ and by $|I|=i_0+\cdots +i_n$. Similarly, if $P$ is a point of $\PP^n$ of coordinates $[p_0:p_1:\cdots :p_n]$, we denote by $P^I$ the monomial $X^I$ evaluated at $P$, that is $p_0^{i_0}p_1^{i_1}\cdots p_n^{i_n}$.

\begin{Def}
Let $f\in k[x_0,\dots, x_n]$ be a homogenous polynomial, of degree $d$, of the form $f=\sum_{|I|=d}\alpha_IX^I$ and consider a point $P\in \PP^n\setminus \Delta_n$. The Hadamard transformation of $f$ by $P$ is the polynomial 
\[
f^{\hada P}=\sum_{|I|=d}\frac{\alpha_I}{P^I}X^I.
\]

\end{Def}

\begin{Thm}\label{genset}
Let $V\subset \PP^n$ be a variety and consider a point $P\in \PP^n\setminus \Delta_n$. If $f_1, \dots, f_s\subset k[x_0, \dots, x_n]$ is a generating set for $I(V)$, that is $I(V)=\langle f_1, \dots, f_s\rangle$, then $f_1^{\hada P}, \dots, f_s^{\hada P}$ is a generating set for $I(P\hada V)$.
\end{Thm}

\begin{Cor}\label{corgenset}
Let $V\subset \PP^n$ be a variety. Then for any point $P\in \PP^n\setminus \Delta_0$ one has $Q\in V$ if and only if $P\hada Q \in P\hada V$.
\end{Cor}


\section{Gorenstein points in $\PP^3$ from the $h$-vector}\label{MNsec}

If $X$ is a subscheme of $\PP^n$ with saturated ideal $I(X)$, and if $t\in \ZZ$ then the Hilbert function of $X$ is denoted by
\[
h_X(t)=dim(k[\PP^{n}]_{t})-dim(I(X)_{t}).
\]
If $X$ is arithmetically Cohen-Macaulay (aCM) of dimension $d$ then $A=k[\PP^{n}]/I(X)$ has Krull dimension $d+1$ and a general set of $d+1$ linear forms forms a regular sequence for $A$. Taking the quotient of $A$ by such a regular sequence gives a zero-dimensional Cohen-Macaulay ring called the {\it Artinian reduction} of $A$. The Hilbert function of the Artinian reduction of $k[\PP^{n}]/I(X)$ is called the {\it $h-$vector} of $X$. This is a finite sequence of integers. The $h-$vector can be also defined as the $(d+1)$-th difference of the Hilbert function of $X$. Thus, when $X$ is a set of points, its $h-$vector is the first difference of its Hilbert function.

Let $n$ and $i$ be positive integers. The $i-${\it binomial expansion of} $n$ is
\[
n_{(i)}=\binom{n_i }{i}+\binom{n_{i -1}}{i-1}+\cdots + \binom{n_j}{j},
\]
where $n_i>n_{i-1} > \cdots > n_j\geq j\geq 1$. The $i-$binomial expansion of $n$ is unique (see, e.g. \cite[Lemma 4.2.6]{BH}). Hence we may define
\[
n_{(i)}^{<i>}=\binom{n_i}{i+1}+\binom{n_{i -1}}{i}+\cdots + \binom{n_j +1}{j+1}.
\]

\begin{Def}
Let ${\mathbf h}=(h_0,h_1,\dots,h_i,\dots)$ be a finite sequence of nonnegative
integers. Then ${\mathbf h}$ is called an {\it O-sequence} if $h_0=1$ and
$h_{i+1}\leq{h_{i}^{<i>}}$ for all $i$.
\end{Def}

By  Macaulay's theorem we know that  $O-$sequences are the Hilbert
functions of standard graded $k$-algebras.
\begin{Def}
Let ${\mathbf h}=(1,h_1,\dots,h_{s-1},1)$ be a sequence of nonnegative
integers. Then ${\mathbf h}$ is an {\it SI-sequence} if:
\begin{itemize}
\item $h_i=h_{s-i}$ for all $i=0,\dots,s$,
\item $(h_0,h_1-h_0,\dots,h_t-h_{t-1},0,\dots)$ is an O-sequence,
where $t$ is the greatest integer $\leq\frac{s}{2}$.
\end{itemize}
\end{Def}

Stanley, in \cite{St}, characterized the $h$-vectors of all graded Artinian 
Gorenstein quotients of $k[x_0,x_1,x_2]$, showing that these are
$SI-$sequence and, moreover,  any $SI-$sequence, with $h_1=3$, is the $h$-vector of
some Artinian Gorenstein quotient of $k[x_0,x_1,x_2]$.

Geramita and Migliore \cite{GM}, show that every minimal
free resolution which occurs for a Gorenstein artinian ideal of
codimension $3$, also occurs for some reduced set of points in $\PP^3$, a
stick figure curve in $\PP^4$ and more generally a ``generalized''
stick figure in $\PP^n$. In this case the points in $\PP^3$,
with such minimal
free resolution, can be found as the intersection of two stick figures (defined below) which are arithmetically Cohen-Macaulay. It is, however, very hard to see
where these points live, that is describe them in term of their coordinates.

We start recall some basic definitions and results that we find in \cite{MN},
\cite{PS}, and \cite{St}.

\begin{Def}
\label{stick figure}
A {\bf generalized stick figure} is a union of linear subvarieties of
$\PP^n$, of the same dimension $d$, such that the intersection of any
three components has dimension at most $d-2$ (the empty set has
dimension -1).
\end{Def}

In particular, sets of reduced  points are stick figure, and a stick
figure of dimension $d=1$ is nothing more than a
reduced union of lines having only nodes as singularities.

\begin{Def} 
Let $C_1$, $C_2$ and $X$ be subschemes of $\PP^{n}$ of the same dimension, where $X$ is a Complete Intersection (arithmetically Gorenstein) such that $I_{X}\subset {I_{C_1}\cap I_{C_2}}$. Then $C_1$ is {\bf directly CI-linked} ({\bf directly G-linked}) to $C_2$ by $X$, if 
\begin{center}
$I_{X}:I_{C_1}=I_{C_2} \mbox{ and } I_{X}:I_{C_2}=I_{C_1}$.
\end{center}
\end{Def}

If $C_1$ is directly linked to $C_2$ by $X$, we will write $C_1\overset{X}{\sim}C_2$ and two schemes $C_1$ and $C_2$ are said to be {\bf residual} to each other.
If, in addition, $C_1$ and $C_2$ have no common components then we say that they are {\bf geometrically linked} by $X$.

There is a important fact that we will use about Liaison: the possibility to built arithmetically Gorenstein zeroscheme starting from two schemes linked by a Complete Intersection. In fact we have the following theorem.

\begin{Thm}[Theorem 4.2.1 in \cite{MiglioreBook}]\label{LiaisonThm}
Let $C_1$, $C_2$ be two aCM subschemes of $\PP^n$ of codimension $c$,  with no common components and saturated ideals $I_{C_1}$ and $I_{C_2}$. If we suppose that $X=C_1\cup C_2$ is a codimension $c$ arithmetically Gorenstein scheme, then $I_{C_1}+I_{C_2}$ is the saturated ideal of a codimension $c+1$ arithmetically Gorenstein scheme $Y$.
\end{Thm}

Now we recall how Migliore and Nagel, in Section 6 of \cite{MN}, find a reduced
arithmetically Gorenstein
zeroscheme, for the case of  $\PP^3$, with given $h-$vector.
This set of points will result from
the intersection of two arithmetically Cohen-Macaulay curves in $\PP^3$,
linked by a complete intersection curve which is a stick figure.

Let
$${\mathbf h}=(h_0 , h_1 , \dots , h_{s})=(1,3,h_2,\dots,h_{t-1},h_t,h_t,\dots,h_t,h_{t-1},\dots,h_2,3,1)$$
be a $SI-$sequence, and consider the first difference
$$\Delta{\mathbf h}=(1,2,h_2-h_1,\dots,h_t-h_{t-1},0,0,\dots,0,h_{t-1}-h_t,\dots,-2,-1).
$$  

Define two sequences ${\mathbf a}=(a_0,\dots,a_t)$ and ${\mathbf g}=(g_0,\dots,g_{s+1})$
in the following way:
\begin{equation}\label{costruzionea}
a_i=h_i-h_{i-1} \mbox{ for } 0\leq{i}\leq{t}
\end{equation}
and 
\begin{equation}\label{costruzioneg}
g_i=
\begin{cases}
i+1 & \hbox{ for } 0\leq{i}\leq{t} \\
t+1 & \hbox{ for } t\leq{i}\leq{s-t+1} \\
s-i+2 & \hbox{ for } s-t+1\leq{i}\leq{s+1} \\
\end{cases}.
\end{equation}

We observe that $a_1=g_1=2$, ${\mathbf a}$ is a $O-$sequence since ${\mathbf h}$ is a
$SI-$sequence and ${\mathbf g}$ is the $h$-vector of a codimension two complete
intersection. So, we would like to find two curves $C_1$ and $X$ 
in $\PP^3$ with $h$-vector respectively ${\mathbf a}$ and ${\mathbf g}$. In particular it is
easy to see that, for that  $h$-vector  ${\mathbf g}$,  $X$ is a complete intersection of
two surfaces in $\PP^3$ of degree $t+1$ and $s-t+2$.

We can get $X$ as a
stick figure by taking, as equations of those surfaces, two polynomials
which are the product, respectively, of
$A_0,\dots,A_{t}$ and $B_0,\dots,B_{s-t+1}$, all
generic linear forms. Considering the entries of ${\mathbf a}=(a_0,\dots a_t)$, Migliore and Nagel build the stick figure $C_1$ (embedded in $X$),
as the union of $a_i$  consecutive lines in $A_i=0$ (always 
the first in $B_0=0$), that is they take $a_0$ lines given by the intersections of  $A_0=0$ with $B_0=0, \dots, B_{a_0-1}=0$, then $a_1$ lines given by the intersections of  $A_1=0$ with $B_0=0, \dots, B_{a_1-1}=0$.
Here {\it consecutive} is referred to the indices of the forms $B_0,\dots,B_{s-t+1}$: two lines are consecutive if they are given by the intersections of a certain $A_i=0$ with $B_j=0$ and $B_{j+1}=0$ for a given $j$ with $0\leq j \leq s-t$.
Migliore and Nagel proved that $C_1$, build in this way, is an aCM scheme  
with $h$-vector ${\mathbf a}$ (Corollary 3.7 in \cite{MN}). In this way, if we consider $C_2$, the residual of
$C_1$ in $X$, the
intersection of $C_1$ and $C_2$ is an arithmetically Gorenstein scheme
$Y$ of codimension $3$, by Theorem \ref{LiaisonThm}. This is also a reduced set of points because
$X$, $C_1$ and $C_2$ are stick figures and it has the desired
$h$-vector by the following theorem:
\begin{Thm}[Lemma 2.5 in \cite{MN}]
\label{hvettore}
Let $C_1$, $C_2$, $X$ and $Y$ be defined as above. Let
${\mathbf g}=(1, c, g_2, \dots, g_{s}, g_{s+1})$ be the h-vector of $X$, and let
${\mathbf a}=(1,a_1,\dots,a_t)$ and ${\mathbf b}=(1,b_1,\dots,b_l)$ be the h-vectors of
$C_1$ and $C_2$, then
$$b_i=g_{s+1-i}-a_{s+1-i}$$
for $i\geq0$. Moreover the sequence $d_i=a_i + b_i -g_i$ is the first
difference of the h-vector ${\mathbf h}=(h_0 , h_1 , \dots , h_{s})$ of $Y$.
\end{Thm}  

As a matter of fact we have $d_i=h_i-h_{i-1}$ since:
\begin{itemize}
\item for $0\leq{i}\leq{t}$ we have $d_i=a_i=h_i-h_{i-1}$;
\item for $t+1\leq{i}\leq{s-t}$ we have $d_i=b_i-g_i=0$;
\item for $s-t+1\leq{i}\leq{s+1}$ we have
$d_i=b_i-g_i=-a_{s+1-i}=-(h_{s+1-i}-h_{s-i})$.
\end{itemize}

\begin{Ex}
\label{punti}
Let ${\mathbf h}=(1,3,4,3,1)$ be a SI-sequence. Consider the first difference of
${\mathbf h}$, i.e. $\Delta{\mathbf h}= (1, 2,  1, -1, -2, -1) $.

So, $t=2$ and ${\mathbf g}=(1,2,3,3,2,1)$ is the $h$-vector of $X$, a stick figure
which is the complete intersection of $F_1=\prod_{i=0}^2{A_i}$ and 
$F_2=\prod_{i=0}^3{B_i}$, where $A_i$ and $B_i$ are general linear
forms.

Now, we call $L_{i,j}$ the intersection between $A_i=0$ and $B_j=0$. Since ${\mathbf a}=(1,2,1)$, then
$C_1=L_{0,0}\cup{L_{1,0}}\cup{L_{1,1}}\cup{L_{2,0}}$ is the scheme, in $X$,
with  $h$-vector ${\mathbf a}$.

So, it is clear that the residual $C_2$ of $C_1$ in $X$ is the
union of the lines of $X$ which aren't components in $C_1$. Then the
reduced set of
points $Y$ with $h$-vector $(1,3,4,3,1)$ consists of $12$ points
which exactly are: 
\begin{itemize}
\item $3$ points on $L_{0,0}$, intersections between $L_{0,0}$ and
$L_{0,1}$, $L_{0,2}$ and $L_{0,3}$;
\item $2$ points on $L_{1,0}$, intersections between $L_{1,0}$ and
$L_{1,2}$, $L_{1,3}$;
\item $4$ points on $L_{1,1}$, intersections between $L_{1,1}$ and
$L_{1,2}$, $L_{1,3}$, $L_{0,1}$ and $L_{2,1}$;
\item $3$ points on $L_{2,0}$, intersections between $L_{2,0}$ and
$L_{2,1}$, $L_{2,2}$ and $L_{2,3}$.
\end{itemize}
\end{Ex}


\section{Planar complete intersections via Hadamard product}\label{planarsec}

In this section we show how to get a zero-dimensional planar complete intersection $Z_{a,b}^\AAA$, as the product of two sets of collinear points.
Observe that, by Corollary 4.5 in \cite{BCFL1}, if the two sets of collinear points lie in two general lines, in $\PP^3$, then their Hadamard product  gives  points on a quadric. However, this could also happen when the lines are coplanar, as explained in the following Remark \ref{BCFLrmk}. Hence, for our  construction of $Z_{a,b}^\AAA$ it is mandatory to carefully choose the coordinates of the points. 
We start by considering four points in $\PP^1$ without zero coordinates.

 Let $\AAA$ be a collection of four distinct points $A_i=[\alpha_i:\beta_i]$ in $\PP^1\setminus \Delta_0$, for $i=0,\dots, 3$, and let
\begin{equation}\label{eqline}
\begin{cases}
\alpha_0x_0+\alpha_1x_1+\alpha_2x_2+\alpha_3x_3=0\\
\beta_0x_0+\beta_1x_1+\beta_2x_2+\beta_3x_3=0
\end{cases}
\end{equation}
be the equations of a line $L^\AAA$ in $\PP^3$.

We define two families of points in $\PP^3$ associated  to the set $\AAA$ (and hence to the line $L^\AAA$):

\[
P^\AAA_k=\left[\frac{\alpha_0+k\beta_0}{\alpha_0}:\frac{\alpha_1+k\beta_1}{\alpha_1}:\frac{\alpha_2+k\beta_2}{\alpha_2}:\frac{\alpha_3+k\beta_3}{\alpha_3}\right] \quad k\in \NN
\]
and
\[
Q^\AAA_k=\left[\frac{k\alpha_0+\beta_0}{\beta_0}:\frac{k\alpha_1+\beta_1}{\beta_1}:\frac{k\alpha_2+\beta_2}{\beta_2}:\frac{k\alpha_3+\beta_3}{\beta_3}\right]\quad k\in \NN.
\]

Note that $P^\AAA_0=Q^\AAA_0=[1:1:1:1]$.

 \begin{Ex}\label{mainex}\rm  
  Consider $A_0=[1:1]$, $A_1=[1:2]$, $A_2=[1:3]$, and $A_3=[1:4]$ giving, by (\ref{eqline}), the line
\[
L^\AAA:\begin{cases}
x_0+x_1+x_2+x_3=0\\
x_0+2x_1+3x_2+4x_3=0
\end{cases}.
\]
One has
 {\small{ \[
  P_1^\AAA=\left[2:3:4:5\right],P_2^\AAA=\left[3:5:7:9\right], P_3^\AAA=\left[4:7:10:13\right],P_4^\AAA=\left[5:9:13:17\right],\dots
\]}}
and
 {\small{ \[
  Q_1^\AAA=\left[2:\frac32:\frac43:\frac54\right],Q_2^\AAA=\left[3:2:\frac53\:\frac{3}{2}\right], Q_3^\AAA=\left[4:\frac52:2:\frac74\right],Q_4^\AAA=\left[5:3:\frac73:2\right],\dots
\]}}  
  
  \end{Ex}

\begin{Rmk}\label{rmkLD}\rm
The condition that the four points $A_i$ are distinct implies  $\frac{\alpha_i}{\beta_i}\not=\frac{\alpha_j}{\beta_j}$ for any $0\leq i < j\leq 3$.
In particular this fact assure us that $L^\AAA\cap \Delta_1=\emptyset$.
As a matter of fact, suppose that, for example, $L^\AAA$ intersects $\Delta_1$ in the point $[0:0:\gamma_2:\gamma_3]$, with $\gamma_i\not=0$, for $i=2,3$.  Then, from (\ref{eqline}), we get
\[
\begin{cases}
\alpha_2\gamma_2+\alpha_3\gamma_3=0\\
\beta_2\gamma_2+\beta_3\gamma_3=0
\end{cases}
\]
which gives
\[
\frac{\alpha_2}{\alpha_3}=-\frac{\gamma_3}{\gamma_2}=\frac{\beta_2}{\beta_3}
\]
or equivalently $\frac{\alpha_2}{\beta_2}=\frac{\alpha_3}{\beta_3}$ which implies $A_2=A_3$.
\end{Rmk}

Notice that 
\[
\begin{array}{l}
P_k^\AAA=[1+k\frac{\beta_0}{\alpha_0}:1+k\frac{\beta_1}{\alpha_1}:1+k\frac{\beta_2}{\alpha_2}:1+k\frac{\beta_3}{\alpha_3}]=\\
\\
=(1-k)[1:1:1:1]+k[1+\frac{\beta_0}{\alpha_0}:1+\frac{\beta_1}{\alpha_1}:1+\frac{\beta_2}{\alpha_2}:1+\frac{\beta_3}{\alpha_3}]=\\
\\
=(1-k)P_0^\AAA+kP_1^\AAA
\end{array}
\]
and similarly $Q_k^\AAA=(1-k)Q_0^\AAA+kQ_1^\AAA$, for all $k\geq 2$. Hence the points $P^\AAA_k$ lie in the line $\ell^P$ spanned by $P_0^\AAA$ and $P_1^\AAA$ and the points $Q^\AAA_k$ lie in the line $\ell^Q$ spanned by $Q_0^\AAA$ and $Q_1^\AAA$.
In particular, for any fixed $k$, the points $P_0^\AAA, \dots , P_k^\AAA$ are collinear and, similarly,  the points $Q_0^\AAA, \dots , Q_k^\AAA$ are collinear.

Consider now the matrices 
\[
M=\begin{pmatrix}
\alpha_0\beta_0 &\alpha_1\beta_1 & \alpha_2\beta_2 & \alpha_3\beta_3\\
\alpha_0^2 &\alpha_1^2 & \alpha_2^2 & \alpha_3^2\\
\beta_0^2 & \beta_1^2 & \beta_2^2 &\beta_3^2
\end{pmatrix}
\quad
N=\begin{pmatrix}
\alpha_0 &\alpha_1 & \alpha_2 & \alpha_3\\
\beta_0 & \beta_1 & \beta_2 &\beta_3
\end{pmatrix}
\]
and denote by $|M(i)|$ the determinant of the submatrix of $M$ with the $i-$th column removed and by  $|N(i,j)|$ the determinant of the submatrix of $N$ with the $i-$th and $j-$th columns removed.

\begin{Prop}\label{lines}
The defining equations in $k[\PP^3]$ of the lines of $\ell^P$ and $\ell^Q$ are:
\[
\ell^P:
\begin{cases}
\sum_{t=0}^3 (-1)^{t+1} \alpha_t\beta_t|M(t+1)|x_t=0\\
\sum_{
t=1}^3(-1)^t\alpha_t|N(1,t+1)|x_t=0

\end{cases},
\]

\[
\ell^Q:
\begin{cases}
\sum_{t=0}^3 (-1)^{t+1} \alpha_t\beta_t|M(t+1)|x_t=0\\
\sum_{
t=1}^3(-1)^t\beta_t|N(1,t+1)|x_t=0

\end{cases}.
\]
Moreover $\ell^P$ and $\ell^Q$ are two distinct coplanar lines.
\end{Prop}

\begin{proof}
We prove the first part of the statement only for $\ell^p$ since the proof is identical for $\ell^Q$.
The equations of $\ell^P$ are given by the equation of the plane through $P_0^\AAA$, $P_1^\AAA$ and $Q_1^\AAA$, 
\[
\left\vert\begin{array}{cccc}
x_0 & x_1 & x_2 & x_3\\
1 & 1 & 1 & 1\\
1+\frac{\beta_0}{\alpha_0} & 1+\frac{\beta_1}{\alpha_1}& 1+\frac{\beta_2}{\alpha_2} & 1+\frac{\beta_3}{\alpha_3}\\
1+\frac{\alpha_0}{\beta_0} & 1+\frac{\alpha_1}{\beta_1} & 1+\frac{\alpha_2}{\beta_2} & 1+\frac{\alpha_3}{\beta_3} \\
\end{array}\right\vert=0
\]
which is, up to rescaling,
\[
\sum_{t=0}^3 (-1)^{t+1} \alpha_t\beta_t|M(t+1)|x_t=0
\]
and by the equation of the plane through $P_0^\AAA$, $P_1^\AAA$ and $[1:0:0:0]$
\[
\left\vert\begin{array}{cccc}
x_0 & x_1 & x_2 & x_3\\
1 & 1 & 1 & 1\\
1+\frac{\beta_0}{\alpha_0} & 1+\frac{\beta_1}{\alpha_1}& 1+\frac{\beta_2}{\alpha_2} & 1+\frac{\beta_3}{\alpha_3}\\
1& 0 & 0 & 0\end{array}\right\vert=0,
\]
which is, up to rescaling
\[
\sum_{t=1}^3(-1)^t\alpha_t|N(1,t+1)|x_t=0.
\]

To prove the second part of the statement notice that $\ell^P$ and $\ell^Q$ intersect at $[1:1:1:1]$. Thus it is enough to prove that $\ell^P$ and $\ell^Q$ are distinct. To this aim, observe that the point $S_P=[\frac{\beta_0}{\alpha_0}:\frac{\beta_1}{\alpha_1}:\frac{\beta_2}{\alpha_2}:\frac{\beta_3}{\alpha_3}]=P_1^\AAA-P_0^\AAA$ lies in $\ell^P$ and the point  $S_Q=[\frac{\alpha_0}{\beta_0}:\frac{\alpha_1}{\beta_1}:\frac{\alpha_2}{\beta_2}:\frac{\alpha_3}{\beta_3}]=Q_1^\AAA-Q_0^\AAA$ lies in $\ell^Q$. Suppose that $\ell^P=\ell^Q$. Then the points $[1:1:1:1]$, $S_P$ and $S_Q$ would be collinear, that is the matrix
\[
\begin{pmatrix}
1 & 1 & 1 & 1\\
\frac{\beta_0}{\alpha_0} &\frac{\beta_1}{\alpha_1} &\frac{\beta_2}{\alpha_2} &\frac{\beta_3}{\alpha_3}\\
\frac{\alpha_0}{\beta_0} &\frac{\alpha_1}{\beta_1} &\frac{\alpha_2}{\beta_2} &\frac{\alpha_3}{\beta_3}
\end{pmatrix}
\]
would have  rank 2.
Applying the operations  $R_2-\frac{\beta_0}{\alpha_0}R_1\to R_2$ and $R_3-\frac{\alpha_0}{\beta_0}R_1\to R_3$ (and then $C_i-C_1$) we get the matrix
\[
\begin{pmatrix}
1 & 0 & 0 & 0\\
0 &\frac{\alpha_0\beta_1-\alpha_1\beta_0}{\alpha_0\alpha_1} &\frac{\alpha_0\beta_2-\alpha_2\beta_0}{\alpha_0\alpha_2} &\frac{\alpha_0\beta_3-\alpha_3\beta_0}{\alpha_0\alpha_3}\\
0 &-\frac{\alpha_0\beta_1-\alpha_1\beta_0}{\beta_0\beta_1} &-\frac{\alpha_0\beta_2-\alpha_2\beta_0}{\beta_0\beta_2} &-\frac{\alpha_0\beta_3-\alpha_3\beta_0}{\beta_0\beta_3}
\end{pmatrix}.
\]
Then we apply the operation $R_3+\frac{\alpha_0\alpha_1}{\beta_0\beta_1}R_2\to R_3$ obtaining
\[
\begin{pmatrix}
1 & 0 & 0 & 0\\
0 &\frac{\alpha_0\beta_1-\alpha_1\beta_0}{\alpha_0\alpha_1} &\frac{\alpha_0\beta_2-\alpha_2\beta_0}{\alpha_0\alpha_2} &\frac{\alpha_0\beta_3-\alpha_3\beta_0}{\alpha_0\alpha_3}\\
0 &0 &\frac{(\alpha_0\beta_2-\alpha_2\beta_0)(\alpha_1\beta_2-\alpha_2\beta_1)}{\alpha_2\beta_0\beta_1\beta_2} &\frac{(\alpha_0\beta_3-\alpha_3\beta_0)(\alpha_1\beta_3-\alpha_3\beta_1)}{\alpha_3\beta_0\beta_1\beta_3}
\end{pmatrix}.
\]
Since  $\frac{\alpha_0}{\beta_0}\not=\frac{\alpha_1}{\beta_1}$, by hypothesis, we can simplify the second row obtaining
\[
\begin{pmatrix}
1 & 0 & 0 & 0\\
0 &\frac{\alpha_0\beta_1-\alpha_1\beta_0}{\alpha_0\alpha_1} &0&0\\
0 &0 &\frac{(\alpha_0\beta_2-\alpha_2\beta_0)(\alpha_1\beta_2-\alpha_2\beta_1)}{\alpha_2\beta_0\beta_1\beta_2} &\frac{(\alpha_0\beta_3-\alpha_3\beta_0)(\alpha_1\beta_3-\alpha_3\beta_1)}{\alpha_3\beta_0\beta_1\beta_3}
\end{pmatrix}.
\]
Thus the matrix would have rank 2 if and only if

\[
(\alpha_0\beta_2-\alpha_2\beta_0)(\alpha_1\beta_2-\alpha_2\beta_1)=0 \mbox{ and } (\alpha_0\beta_3-\alpha_3\beta_0)(\alpha_1\beta_3-\alpha_3\beta_1)=0.
\]
Such equalities  are verified in the following cases
\begin{itemize}
\item $\frac{\alpha_0}{\beta_0}=\frac{\alpha_2}{\beta_2}=\frac{\alpha_3}{\beta_3}$, 
\item $\frac{\alpha_1}{\beta_1}=\frac{\alpha_2}{\beta_2}=\frac{\alpha_3}{\beta_3}$,
\item $\frac{\alpha_0}{\beta_0}=\frac{\alpha_3}{\beta_3}$  and $\frac{\alpha_1}{\beta_1}=\frac{\alpha_2}{\beta_2}$,
\item $\frac{\alpha_0}{\beta_0}=\frac{\alpha_2}{\beta_2}$ and $\frac{\alpha_1}{\beta_1}=\frac{\alpha_3}{\beta_3}$,
\end{itemize}
which give contradictions since the points $A_i$ are distinct.
\end{proof}

By the previous proposition, we immediately get the following

\begin{Cor}\label{distinctpoints}
One has $P_i^\AAA\not= Q_j^\AAA$, for every $i,j\geq 1$.
\end{Cor}
 
 \begin{proof}
 Suppose that, for some $i,j\geq 1$ one has $P_i^\AAA= Q_j^\AAA$. Then $\ell^P$ and $\ell^Q$ would intersect in the points  $[1,1,1,1]$ and $P_i^\AAA(= Q_j^\AAA)$ giving $\ell^P=\ell^Q$, which is a contradiction, by Proposition \ref{lines}.
 \end{proof}
  
\begin{Ex}\label{conex}\rm
Consider  Example \ref{mainex}. In this case one has
\[
M=\begin{pmatrix}
1 & 2 & 3 & 4\cr
1 & 1 & 1 & 1\cr
1 & 4 & 9 & 16
\end{pmatrix}
\qquad
N=\begin{pmatrix}
1 & 1 & 1 & 1\cr
1 & 2 & 3 & 4\cr
\end{pmatrix}
\]
from which we get
\[
\begin{array}{llll}
|M(1)|:=-2 & |M(2)|:=-6 &|M(3)|:=-6 & |M(4)|:=-2\\ 
\end{array}
\]
and
\[
\begin{array}{llll}
|N(1,2)|:=1 & |N(1,3)|:=2 &|N(1,4)|:=1 
\end{array}.
\]
Hence the line $\ell^P$ through the points $P_k^\AAA$ is defined, up to rescaling, by the equations
\[
\ell^P:
\begin{cases}
2x_0-12x_1+18x_2-8x_3=0\\
-x_1+2x_2-x_3=0
\end{cases}
\]
and the line $\ell^Q$ through the points $Q_k^\AAA$ is defined, up to rescaling, by the equations
\[
\ell^Q:
\begin{cases}
2x_0-12x_1+18x_2-8x_3=0\\
-2x_1+6x_2-4x_3=0
\end{cases}.
\]

\end{Ex}

 \vskip0.3cm
 We add, now, another condition on the points in $\AAA$.
Let $W_{i}$ be the point  $[1:-i]$, then we define the set of points $\mathcal{W}$ as
 \[
\mathcal{W}=\bigcup_{i\in \NN^*} \left(W_i\cup W_\frac{1}{i}\right)
  \]
  where $\NN^*=\NN\setminus \{ 0\}$.
  
\begin{Rmk}\label{remarkzeri}\rm
 It is easy to verify that if $A_i\notin \mathcal{W}$, for $i=0,\dots, 3$, then $P_i^\AAA\notin \Delta_2$, for any $i$, and $Q_j^\AAA\notin \Delta_2$, for any $j$, that is such points do not have any zero coordinate. This fact will be fundamental in the successive parts of the paper in order to apply Lemmas \ref{LempL}, \ref{LempqL} and \ref{Lempqr}. Clearly, $A_i\notin \mathcal{W}$ if its coordinates are both strictly positive. 
\end{Rmk}

Denote by ${\mathcal{I}(n})=\{i_0, i_1, \dots, i_{n-1}\}$ a set of nonnegative integers with $0=i_0<i_1<\cdots <i_{n-1}$.
Given positive integers $a$ and $b$, we define the set of points $Z^\AAA_{a,b}$ by the pair-wise Hadamard product of points  $P_i^\AAA$ and $Q_j^\AAA$ as
\[
 Z^\AAA_{a,b}=\{P^\AAA_i\hada Q^\AAA_j\, : \, i\in {\mathcal{I}(a}),  \, j\in {\mathcal{I}(b})\}.
 \]
We can represent these sets in matrix form as:
 \[
 \begin{pmatrix}
P^\AAA_{i_0} \hada Q^\AAA_{i_0} & P^\AAA_{i_0} \hada Q^\AAA_{i_1} & \cdots & P^\AAA_{i_0}\hada Q^\AAA_{i_{b-1}}\cr
P^\AAA_{i_1} \hada Q^\AAA_{i_0} & P^\AAA_{i_1} \hada Q^\AAA_{i_1} & \cdots & P^\AAA_{i_1}\hada Q^\AAA_{i_{b-1}}\cr
\vdots & \vdots & \ddots & \vdots\cr
P^\AAA_{i_{a-1}} \hada Q^\AAA_{i_0} & P^\AAA_{i_{a-1}} \hada Q^\AAA_{i_1} & \cdots & P^\AAA_{i_{a-1}}\hada Q^\AAA_{i_{b-1}}\
\end{pmatrix}.
\]
Observe that, by the conditions on ${\mathcal{I}(a})$ and ${\mathcal{I}(b})$ one has 
\[
P^\AAA_{i_0}=P^\AAA_{0}=[1:1:1:1]=Q^\AAA_{0}=Q^\AAA_{i_0}.
\]

\begin{Thm}\label{IC}
 If $A_i\notin \mathcal{W}$, for $i=0,\dots, 3$, then, for any positive integers $a$ and $b$, $Z_{a,b}^\AAA$ is a planar complete intersection of $ab$ points.
\end{Thm}

\begin{proof}
Consider $i,k\in {\mathcal{I}(a})$ and $j,l \in {\mathcal{I}(b})$. We prove first that  $P_i^\AAA\hada Q_j^\AAA= P_k^\AAA\hada Q_l^\AAA$ if and only if $i=k$ and $j=l$, implying  that $Z_{a,b}^\AAA$ is a set of cardinality $ab$.

Suppose that $P_i^\AAA\hada Q_j^\AAA= P_k^\AAA\hada Q_l^\AAA$ and  distinguish two cases.
First, we consider the case  in which two indices are equal. Suppose, for example, that $i=k$ and $j\not=l$, i.e. $P_i^\AAA\hada Q_j^\AAA= P_i^\AAA\hada Q_l^\AAA$. Since, by Remark \ref{remarkzeri}, $P_i^\AAA\notin \Delta_2$, one has, by Lemma \ref{Lempqr}, that $Q_j^\AAA= Q_l^\AAA$, which is a contradiction since $j\not=l$. The same approach works if $i\not=k$ and $j=l$.
Let us consider the case $i\not=k$ and $j\not=l$. Looking at the coordinates, the condition $P_i^\AAA\hada Q_j^\AAA= P_k^\AAA\hada Q_l^\AAA$, is
{\small{
\[
\begin{array}{lll}
&\left[ \frac{(\alpha_0+i\beta_0)(j\alpha_0+\beta_0)}{\alpha_0\beta_0}: \frac{(\alpha_1+i\beta_1)(j\alpha_1+\beta_1)}{\alpha_1\beta_1}:\frac{(\alpha_2+i\beta_2)(j\alpha_2+\beta_2)}{\alpha_2\beta_2}: \frac{(\alpha_3+i\beta_3)(j\alpha_3+\beta_3)}{\alpha_3\beta_3}\right]&=\\
\\
=&\left[ \frac{(\alpha_0+k\beta_0)(l\alpha_0+\beta_0)}{\alpha_0\beta_0}: \frac{(\alpha_1+k\beta_1)(l\alpha_1+\beta_1)}{\alpha_1\beta_1}:\frac{(\alpha_2+k\beta_2)(l\alpha_2+\beta_2)}{\alpha_2\beta_2}: \frac{(\alpha_3+k\beta_3)(l\alpha_3+\beta_3)}{\alpha_3\beta_3}\right]
\end{array}
\]}}
or equivalently
\[
 \frac{(\alpha_s+i\beta_s)(j\alpha_s+\beta_s)}{(\alpha_s+k\beta_s)(l\alpha_s+\beta_s)}=\lambda \mbox{ for } s=0,\dots, 3
\]
for some $\lambda\not=0$. This implies
\[
 \frac{(\alpha_0+i\beta_0)(j\alpha_0+\beta_0)}{(\alpha_0+k\beta_0)(l\alpha_0+\beta_0)}= \frac{(\alpha_s+i\beta_s)(j\alpha_s+\beta_s)}{(\alpha_s+k\beta_s)(l\alpha_s+\beta_s)} \mbox{ for } s=1,\dots, 3.
\]
Hence $[\alpha_0:\beta_0]$, $[\alpha_1:\beta_1]$, $[\alpha_2:\beta_2]$ and $[\alpha_3:\beta_3]$ must satisfy

{\small{\begin{equation}\label{eqbr}
(\alpha_0+i\beta_0)(j\alpha_0+\beta_0)(\alpha_s+k\beta_s)(l\alpha_s+\beta_s)-(\alpha_s+i\beta_s)(j\alpha_s+\beta_s)(\alpha_0+k\beta_0)(l\alpha_0+\beta_0)=0
\end{equation}}}

\noindent for $s=1,\dots, 3$. If we rewrite (\ref{eqbr}) as an equation in $\alpha_s$, we get $\tau_2\alpha_s^2+\tau_1\alpha_s+\tau_0=0$ where
\begin{equation*}
\begin{array}{rcl}
\tau_2&=&(\alpha_0+i\beta_0)(j\alpha_0+\beta_0)l-(\alpha_0+k\beta_0)(l\alpha_0+\beta_0)j,\\
\\
\tau_1&=&[(\alpha_0+i\beta_0)(j\alpha_0+\beta_0)(kl+1)-(\alpha_0+k\beta_0)(l\alpha_0+\beta_0)(ij+1)]\beta_s,\\
\\
\tau_0&=&[(\alpha_0+i\beta_0)(j\alpha_0+\beta_0)k-(\alpha_0+k\beta_0)(l\alpha_0+\beta_0)i]\beta_s^2.
\end{array}
\end{equation*}

The discriminant of $\tau_2\alpha_s^2+\tau_1\alpha_s+\tau_0$ turns to be equal to

\[
\beta_s^2(j\alpha_0^2-l\alpha_0^2-ijl\alpha_0^2+jkl\alpha_0^2+2jk\alpha_0\beta_0-2il\alpha_0\beta_0-i\beta_0^2+k\beta_0^2+ijk\beta_0^2-ikl\beta_0^2)^2
\]

\noindent which gives, after some tedious computation, the solutions of $\alpha_s$ as
\begin{equation*}
\alpha_s=\frac{\alpha_0\beta_s}{\beta_0} \mbox { or } \alpha_s =\rho \beta_s, \mbox{ for } s=1,\dots, 3,
\end{equation*}
where 
\[
\rho=\frac{(jk-il)\alpha_0+(ijk-i-ikl+k)\beta_0}{(ijl-j+l-jkl)\alpha_0+(il-jk)\beta_0}.
\]
Computing the solutions of (\ref{eqbr}) for $s=1,\dots, 3$, we obtain that $P_i^\AAA\hada Q_j^\AAA= P_k^\AAA\hada Q_l^\AAA$ if one of the following cases is verified
\begin{itemize}
\item[i)] $\alpha_1=\frac{\alpha_0\beta_1}{\beta_0}$, $\alpha_2=\frac{\alpha_0\beta_2}{\beta_0}$, $\alpha_3=\frac{\alpha_0\beta_3}{\beta_0}$;
\item[ii)] $\alpha_1=\frac{\alpha_0\beta_1}{\beta_0}$, $\alpha_2=\frac{\alpha_0\beta_2}{\beta_0}$, $\alpha_3=\rho \beta_3$;
\item[iii)] $\alpha_1=\frac{\alpha_0\beta_1}{\beta_0}$, $\alpha_2=\rho \beta_2$, $\alpha_3=\frac{\alpha_0\beta_3}{\beta_0}$;
\item[vi)] $\alpha_1=\frac{\alpha_0\beta_1}{\beta_0}$, $\alpha_2=\rho \beta_2$, $\alpha_3=\rho \beta_3$;
\item[v)] $\alpha_1=\rho \beta_1$, $\alpha_2=\frac{\alpha_0\beta_2}{\beta_0}$, $\alpha_3=\frac{\alpha_0\beta_3}{\beta_0}$;
\item[vi)] $\alpha_1=\rho \beta_1$, $\alpha_2=\frac{\alpha_0\beta_2}{\beta_0}$, $\alpha_3=\rho \beta_3$;
\item[vii)] $\alpha_1=\rho \beta_1$, $\alpha_2=\rho \beta_2$, $\alpha_3=\frac{\alpha_0\beta_3}{\beta_0}$;
\item[viii)] $\alpha_1=\rho \beta_1$, $\alpha_2=\rho \beta_2$, $\alpha_3=\rho \beta_3$.
\end{itemize}
However, all cases implies that there are at least  two pairs of indices $(\rho_1,\rho_2)$ and $(\rho_3,\rho_4)$ with $\frac{\alpha_{\rho_1}}{\beta_{\rho_1}}=\frac{\alpha_{\rho_2}}{\beta_{\rho_2}}$ and $\frac{\alpha_{\rho_3}}{\beta_{\rho_3}}=\frac{\alpha_{\rho_4}}{\beta_{\rho_4}}$  which is a contradiction since the points $A_i$ must be distinct.  Hence $P_i^\AAA\hada Q_j^\AAA= P_k^\AAA\hada Q_l^\AAA$ if and only if $i=k$ and $j=l$ and then $Z_{a,b}^\AAA$ consists of $ab$ points.

To prove that $Z_{a,b}^\AAA$ is a planar complete intersection, notice first that, since,  $P_i^\AAA\notin \Delta_2$, for $i\in {\mathcal{I}(a})$ and $Q_j^\AAA\notin \Delta_2$ for $ j\in {\mathcal{I}(b})$, we can apply Lemma \ref{LempL} obtaining that
$P_i^\AAA\hada \ell^Q$ is a line for $i\in {\mathcal{I}(a})$ and $Q_j^\AAA\hada \ell^P$ is a line for $ j\in {\mathcal{I}(b})$. Moreover, by Corollary \ref{corgenset}, one has
\[
P_i^\AAA\hada Q_j^\AAA \in Q_j^\AAA\hada \ell^P \mbox { for } i\in {\mathcal{I}(a})
\]
\[
P_i^\AAA\hada Q_j^\AAA \in P_i^\AAA\hada \ell^Q \mbox { for }  j\in {\mathcal{I}(b})
\]
\vskip0.1cm
\noindent that is the points $P_{i_0}^\AAA\hada Q_j^\AAA, \dots, P_{i_{a-1}}^\AAA\hada Q_j^\AAA$ lie in the line $Q_j\hada \ell^P$, for $ j\in {\mathcal{I}(b})$ and similarly the points $P_i^\AAA\hada Q_{i_0}^\AAA, \dots, P_i^\AAA\hada Q_{i_{b-1}}^\AAA$ lie in the line $P_i\hada \ell^Q$, for $i\in {\mathcal{I}(a})$.

For any $i$ and $j$ the lines $P_i\hada \ell^Q$ and $Q_j\hada \ell^P$ clearly  intersect in the point $P^\AAA_i\hada Q_j^\AAA$, hence, as $i$ varies in ${\mathcal{I}(a})$ and $j$ varies in ${\mathcal{I}(b})$, the $ab$ intersections $P_i\hada \ell^Q\cap Q_j\hada \ell^P$ give the $ab$ points in $Z_{a,b}^\AAA$.
\end{proof}

In Figure \ref{figuraZ} we can see four different examples of $Z^\AAA_{a,b}$. The example in $(i)$ is  for $a=b=2$ with ${\mathcal{I}(a})={\mathcal{I}(b})=\{0,1\}$ and the white points are represented to show the behaviour of the families of points $P^\AAA$ and $Q^\AAA$. The example in $(ii)$ is for $a=4$ and $b=5$ with ${\mathcal{I}(a})=\{0,1,2,3\}$ and ${\mathcal{I}(b})=\{0,1,2,3,4\}$. The examples in $(iii)$ and $(iv)$ are for  $a=2$ and $b=3$, but, while in $(iii)$ we use ${\mathcal{I}(a})=\{0,1\}$ and ${\mathcal{I}(b})=\{0,1,2\}$, in $(iv)$ we use ${\mathcal{I}(a})=\{0,2\}$ and ${\mathcal{I}(b})=\{0,2,4\}$.

\begin{figure}
\centering
\begin{tabular}{cc}
\begin{tikzpicture}
   [scale=0.5,auto=center]
   
\draw (-1.5,5) -- (9,5);
\draw[dotted] (9,5) -- (10,5);
\draw (0,-1.5) -- (0,6.5);
\draw[dotted] (0,-1.5) -- (0,-2.5);

\draw(0.7,6.3) node[scale=0.5] {$\ell_P\star Q_0^\AAA$};  
\draw(-0.8,5.3) node[scale=0.5] {$\ell_Q\star P_0^\AAA$};  
\draw[black] (-1.5,3) -- (9,3);
\draw[black,dotted] (9,3) -- (10,3);
\draw[black] (2,-1.5) -- (2,6.5);
\draw[black,dotted] (2,-1.5) -- (2,-2.5);
\draw(2.7,6.3) node[scale=0.5] {$\ell_P\star Q_1^\AAA$};  
\draw(-0.8,3.3) node[scale=0.5] {$\ell_Q\star P_1^\AAA$}; 
\draw(0.85,5.3) node[scale=0.5] {$Q_0^\AAA\star P_0^\AAA$};  
\draw(2.85,5.3) node[scale=0.5] {$Q_1 ^\AAA\star P_0^\AAA$};  
\draw(4.5,5.3) node[scale=0.5] {$Q_2^\AAA$};  
\draw(6.5,5.3) node[scale=0.5] {$Q_3^\AAA$};  
\draw(8.5,5.3) node[scale=0.5] {$Q_4^\AAA$};  
\draw(0.85,3.3) node[scale=0.5] {$Q_0^\AAA \star P_1^\AAA$};  
\draw(2.85,3.3) node[scale=0.5] {$Q_1^\AAA \star P_1^\AAA$};  
\draw(0.5,1.3) node[scale=0.5] {$P_2^\AAA$};  
\draw(0.5,-0.7) node[scale=0.5] {$P_3^\AAA$};  

  \node[circle, draw=black!100,fill=black!100,inner sep=2pt] (n1) at (0,5) {};
   \node[circle, draw=black!100,fill=black!100,inner sep=2pt] (n1) at (2,5) {};
   \node[circle, draw=black!100,fill=white!100,inner sep=2pt] (n1) at (4,5) {};
   \node[circle, draw=black!100,fill=white!100,inner sep=2pt] (n1) at (6,5) {};
   \node[circle, draw=black!100,fill=white!100,inner sep=2pt] (n1) at (8,5) {};

  \node[circle, draw=black!100,fill=black!100,inner sep=2pt] (n1) at (0,3) {};
  \node[circle, draw=black!100,fill=black!100,inner sep=2pt] (n1) at (2,3) {};
  \node[circle, draw=black!100,fill=white!100,inner sep=2pt] (n1) at (0,1) {};
  \node[circle, draw=black!100,fill=white!100,inner sep=2pt] (n1) at (0,-1) {};

\end{tikzpicture}
&
\begin{tikzpicture}
   [scale=0.5,auto=center]
   
\draw (-1.5,5) -- (9,5);
\draw[dotted] (9,5) -- (10,5);
\draw (0,-1.5) -- (0,6.5);
\draw[dotted] (0,-1.5) -- (0,-2.5);

\draw(0.7,6.3) node[scale=0.5] {$\ell_P\star Q_0^\AAA$};  
\draw(-0.8,5.3) node[scale=0.5] {$\ell_Q\star P_0^\AAA$};  
\draw[black] (-1.5,3) -- (9,3);
\draw[black,dotted] (9,3) -- (10,3);
\draw[black] (2,-1.5) -- (2,6.5);
\draw[black,dotted] (2,-1.5) -- (2,-2.5);
\draw(2.7,6.3) node[scale=0.5] {$\ell_P\star Q_1^\AAA$};  
\draw(-0.8,3.3) node[scale=0.5] {$\ell_Q\star P_1^\AAA$};  

\draw[black] (4,-1.5) -- (4,6.5);
\draw[black,dotted] (4,-1.5) -- (4,-2.5);
\draw(4.7,6.3) node[scale=0.5] {$\ell_P\star Q_2^\AAA$};

\draw[black] (6,-1.5) -- (6,6.5);
\draw[black,dotted] (6,-1.5) -- (6,-2.5);
\draw(6.7,6.3) node[scale=0.5] {$\ell_P\star Q_3^\AAA$};  

\draw[black] (8,-1.5) -- (8,6.5);
\draw[black,dotted] (8,-1.5) -- (8,-2.5);
\draw(8.7,6.3) node[scale=0.5] {$\ell_P\star Q_4^\AAA$};

\draw[black] (-1.5,1) -- (9,1);
\draw[black,dotted] (9,1) -- (10,1);
\draw(-0.8,1.3) node[scale=0.5] {$\ell_Q\star P_2^\AAA$};  

\draw[black] (-1.5,-1) -- (9,-1);
\draw[black,dotted] (9,-1) -- (10,-1);
\draw(-0.8,-0.7) node[scale=0.5] {$\ell_Q\star P_3^\AAA$}; 

\draw(0.85,5.3) node[scale=0.5] {$Q_0^\AAA\star P_0^\AAA$};  
\draw(2.85,5.3) node[scale=0.5] {$Q_1^\AAA \star P_0^\AAA$};  
\draw(4.85,5.3) node[scale=0.5] {$Q_2^\AAA \star P_0^\AAA$};  
\draw(6.85,5.3) node[scale=0.5] {$Q_3^\AAA \star P_0^\AAA$};  
\draw(8.85,5.3) node[scale=0.5] {$Q_4^\AAA \star P_0^\AAA$};  
\draw(0.85,3.3) node[scale=0.5] {$Q_0^\AAA \star P_1^\AAA$};  
\draw(2.85,3.3) node[scale=0.5] {$Q_1^\AAA \star P_1^\AAA$};  
\draw(4.85,3.3) node[scale=0.5] {$Q_2^\AAA \star P_1^\AAA$}; 
\draw(6.85,3.3) node[scale=0.5] {$Q_3^\AAA \star P_1^\AAA$}; 
\draw(8.85,3.3) node[scale=0.5] {$Q_4^\AAA \star P_1^\AAA$}; 

\draw(0.85,1.3) node[scale=0.5] {$Q_0^\AAA \star P_2^\AAA$};  
\draw(2.85,1.3) node[scale=0.5] {$Q_1^\AAA \star P_2^\AAA$};  
\draw(4.85,1.3) node[scale=0.5] {$Q_2^\AAA \star P_2^\AAA$}; 
\draw(6.85,1.3) node[scale=0.5] {$Q_3^\AAA \star P_2^\AAA$}; 
\draw(8.85,1.3) node[scale=0.5] {$Q_4^\AAA \star P_2^\AAA$}; 

\draw(0.85,-0.7) node[scale=0.5] {$Q_0^\AAA \star P_3^\AAA$};  
\draw(2.85,-0.7) node[scale=0.5] {$Q_1^\AAA \star P_3^\AAA$};  
\draw(4.85,-0.7) node[scale=0.5] {$Q_2^\AAA \star P_3^\AAA$}; 
\draw(6.85,-0.7) node[scale=0.5] {$Q_3^\AAA \star P_3^\AAA$}; 
\draw(8.85,-0.7) node[scale=0.5] {$Q_4^\AAA \star P_3^\AAA$};

  \node[circle, draw=black!100,fill=black!100,inner sep=2pt] (n1) at (0,5) {};
   \node[circle, draw=black!100,fill=black!100,inner sep=2pt] (n1) at (2,5) {};
   \node[circle, draw=black!100,fill=black!100,inner sep=2pt] (n1) at (4,5) {};
   \node[circle, draw=black!100,fill=black!100,inner sep=2pt] (n1) at (6,5) {};
   \node[circle, draw=black!100,fill=black!100,inner sep=2pt] (n1) at (8,5) {};

  \node[circle, draw=black!100,fill=black!100,inner sep=2pt] (n1) at (0,3) {};
   \node[circle, draw=black!100,fill=black!100,inner sep=2pt] (n1) at (2,3) {};
   \node[circle, draw=black!100,fill=black!100,inner sep=2pt] (n1) at (4,3) {};
   \node[circle, draw=black!100,fill=black!100,inner sep=2pt] (n1) at (6,3) {};
   \node[circle, draw=black!100,fill=black!100,inner sep=2pt] (n1) at (8,3) {};

  \node[circle, draw=black!100,fill=black!100,inner sep=2pt] (n1) at (0,1) {};
   \node[circle, draw=black!100,fill=black!100,inner sep=2pt] (n1) at (2,1) {};
   \node[circle, draw=black!100,fill=black!100,inner sep=2pt] (n1) at (4,1) {};
   \node[circle, draw=black!100,fill=black!100,inner sep=2pt] (n1) at (6,1) {};
   \node[circle, draw=black!100,fill=black!100,inner sep=2pt] (n1) at (8,1) {};

  \node[circle, draw=black!100,fill=black!100,inner sep=2pt] (n1) at (0,-1) {};
   \node[circle, draw=black!100,fill=black!100,inner sep=2pt] (n1) at (2,-1) {};
   \node[circle, draw=black!100,fill=black!100,inner sep=2pt] (n1) at (4,-1) {};
   \node[circle, draw=black!100,fill=black!100,inner sep=2pt] (n1) at (6,-1) {};
   \node[circle, draw=black!100,fill=black!100,inner sep=2pt] (n1) at (8,-1) {};

\end{tikzpicture}\\
$(i)$ & $(ii)$\\
\\
\begin{tikzpicture}
   [scale=0.5,auto=center]
   
\draw (-1.5,5) -- (9,5);
\draw[dotted] (9,5) -- (10,5);
\draw (0,-1.5) -- (0,6.5);
\draw[dotted] (0,-1.5) -- (0,-2.5);

\draw(0.7,6.3) node[scale=0.5] {$\ell_P\star Q_0^\AAA$};  
\draw(-0.8,5.3) node[scale=0.5] {$\ell_Q\star P_0^\AAA$};  
\draw[black] (-1.5,3) -- (9,3);
\draw[black,dotted] (9,3) -- (10,3);
\draw[black] (2,-1.5) -- (2,6.5);
\draw[black,dotted] (2,-1.5) -- (2,-2.5);
\draw(2.7,6.3) node[scale=0.5] {$\ell_P\star Q_1^\AAA$};  
\draw(-0.8,3.3) node[scale=0.5] {$\ell_Q\star P_1^\AAA$};  

\draw[black] (4,-1.5) -- (4,6.5);
\draw[black,dotted] (4,-1.5) -- (4,-2.5);
\draw(4.7,6.3) node[scale=0.5] {$\ell_P\star Q_2^\AAA$};

\draw[black] (6,-1.5) -- (6,6.5);
\draw[black,dotted] (6,-1.5) -- (6,-2.5);
\draw(6.7,6.3) node[scale=0.5] {$\ell_P\star Q_3^\AAA$};  

\draw[black] (8,-1.5) -- (8,6.5);
\draw[black,dotted] (8,-1.5) -- (8,-2.5);
\draw(8.7,6.3) node[scale=0.5] {$\ell_P\star Q_4^\AAA$};

\draw[black] (-1.5,1) -- (9,1);
\draw[black,dotted] (9,1) -- (10,1);
\draw(-0.8,1.3) node[scale=0.5] {$\ell_Q\star P_2^\AAA$};  

\draw[black] (-1.5,-1) -- (9,-1);
\draw[black,dotted] (9,-1) -- (10,-1);
\draw(-0.8,-0.7) node[scale=0.5] {$\ell_Q\star P_3^\AAA$}; 

\draw(0.85,5.3) node[scale=0.5] {$Q_0^\AAA\star P_0^\AAA$};  
\draw(2.85,5.3) node[scale=0.5] {$Q_1^\AAA \star P_0^\AAA$};  
\draw(4.85,5.3) node[scale=0.5] {$Q_2^\AAA \star P_0^\AAA$};  

\draw(0.85,3.3) node[scale=0.5] {$Q_0^\AAA \star P_1^\AAA$};  
\draw(2.85,3.3) node[scale=0.5] {$Q_1^\AAA \star P_1^\AAA$};  
\draw(4.85,3.3) node[scale=0.5] {$Q_2^\AAA \star P_1^\AAA$};

  \node[circle, draw=black!100,fill=black!100,inner sep=2pt] (n1) at (0,5) {};
   \node[circle, draw=black!100,fill=black!100,inner sep=2pt] (n1) at (2,5) {};
   \node[circle, draw=black!100,fill=black!100,inner sep=2pt] (n1) at (4,5) {};

  \node[circle, draw=black!100,fill=black!100,inner sep=2pt] (n1) at (0,3) {};
   \node[circle, draw=black!100,fill=black!100,inner sep=2pt] (n1) at (2,3) {};
   \node[circle, draw=black!100,fill=black!100,inner sep=2pt] (n1) at (4,3) {};

\end{tikzpicture} & \begin{tikzpicture}
   [scale=0.5,auto=center]
   
\draw (-1.5,5) -- (9,5);
\draw[dotted] (9,5) -- (10,5);
\draw (0,-1.5) -- (0,6.5);
\draw[dotted] (0,-1.5) -- (0,-2.5);

\draw(0.7,6.3) node[scale=0.5] {$\ell_P\star Q_0^\AAA$};  
\draw(-0.8,5.3) node[scale=0.5] {$\ell_Q\star P_0^\AAA$};  
\draw[black] (-1.5,3) -- (9,3);
\draw[black,dotted] (9,3) -- (10,3);
\draw[black] (2,-1.5) -- (2,6.5);
\draw[black,dotted] (2,-1.5) -- (2,-2.5);
\draw(2.7,6.3) node[scale=0.5] {$\ell_P\star Q_1^\AAA$};  
\draw(-0.8,3.3) node[scale=0.5] {$\ell_Q\star P_1^\AAA$};  

\draw[black] (4,-1.5) -- (4,6.5);
\draw[black,dotted] (4,-1.5) -- (4,-2.5);
\draw(4.7,6.3) node[scale=0.5] {$\ell_P\star Q_2^\AAA$};

\draw[black] (6,-1.5) -- (6,6.5);
\draw[black,dotted] (6,-1.5) -- (6,-2.5);
\draw(6.7,6.3) node[scale=0.5] {$\ell_P\star Q_3^\AAA$};  

\draw[black] (8,-1.5) -- (8,6.5);
\draw[black,dotted] (8,-1.5) -- (8,-2.5);
\draw(8.7,6.3) node[scale=0.5] {$\ell_P\star Q_4^\AAA$};

\draw[black] (-1.5,1) -- (9,1);
\draw[black,dotted] (9,1) -- (10,1);
\draw(-0.8,1.3) node[scale=0.5] {$\ell_Q\star P_2^\AAA$};  

\draw[black] (-1.5,-1) -- (9,-1);
\draw[black,dotted] (9,-1) -- (10,-1);
\draw(-0.8,-0.7) node[scale=0.5] {$\ell_Q\star P_3^\AAA$}; 

\draw(0.85,5.3) node[scale=0.5] {$Q_0^\AAA\star P_0^\AAA$};  

\draw(4.85,5.3) node[scale=0.5] {$Q_2^\AAA \star P_0^\AAA$};  

\draw(8.85,5.3) node[scale=0.5] {$Q_4^\AAA \star P_0^\AAA$};

\draw(0.85,1.3) node[scale=0.5] {$Q_0^\AAA \star P_2^\AAA$};  

\draw(4.85,1.3) node[scale=0.5] {$Q_2^\AAA \star P_2^\AAA$}; 

\draw(8.85,1.3) node[scale=0.5] {$Q_4^\AAA \star P_2^\AAA$};

  \node[circle, draw=black!100,fill=black!100,inner sep=2pt] (n1) at (0,5) {};

   \node[circle, draw=black!100,fill=black!100,inner sep=2pt] (n1) at (4,5) {};

   \node[circle, draw=black!100,fill=black!100,inner sep=2pt] (n1) at (8,5) {};

  \node[circle, draw=black!100,fill=black!100,inner sep=2pt] (n1) at (0,1) {};

   \node[circle, draw=black!100,fill=black!100,inner sep=2pt] (n1) at (4,1) {};

   \node[circle, draw=black!100,fill=black!100,inner sep=2pt] (n1) at (8,1) {};

\end{tikzpicture}\\
$(iii)$ & $(iv)$\\
\end{tabular}
\caption{Four examples of  $Z^\AAA_{a,b}$ for different choices of ${\mathcal{I}}(a)$ and ${\mathcal{I}}(b)$.}\label{figuraZ}
\end{figure}
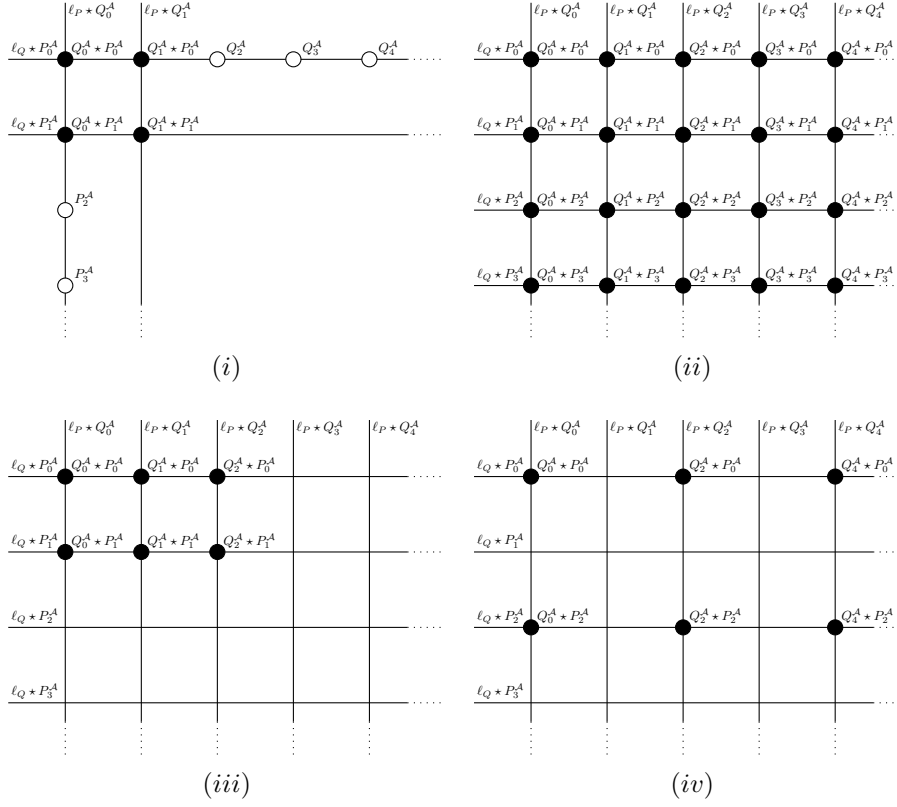

\begin{Rmk}\label{BCFLrmk}\rm
In \cite{BCFL1}, the authors prove that the Hadamard product of two generic lines is a quadric surface. This leads to the question if the Hadamard product of two coplanar lines is a plane. The following example shows that even when the lines $\ell$ and $\ell'$ are coplanar, $\ell \star \ell'$  might  still  be  a quadric. Consider the points
\[
S_1=[1,1,1,1], \,S_2=[3,\frac32, 5,\frac72], \,S_3=[\frac32, 3, \frac43, \frac75]
\]
and the coplanar lines
\[\ell=\overline{S_1S_2} \quad \ell'=\overline{S_1S_3}.\]
One has
\[
\ell:\begin{cases}
3x_1+4x_2-7x_3=0\\
7x_0-4x_1-3x_2=0
\end{cases}
\]
\[
\ell':\begin{cases}
x_1+24x_2-25x_3=0\\
10x_0-x_1-9x_2=0
\end{cases}.
\]

Using the \texttt{Singular} procedure \texttt{HPr}, described in Section 5 of \cite{BCFL1}, we can easily see that $\ell\star \ell'$ is a quadric: 

\begin{verbatim}
> ring R=0,(x(0..3)),dp;
> ideal J1=3*x(1)+4*x(2)-7*x(3),7*x(0)-4*x(1)-3*x(2);
> ideal J2=x(1)+24*x(2)-25*x(3), 10*x(0)-x(1)-9*x(2);
> ideal K=HPr(J1,J2,3);
> K;
K[1]=1120*x(0)^2-68*x(0)*x(1)+x(1)^2+1056*x(0)*x(2)-
-30*x(1)*x(2)+216*x(2)^2-3500*x(0)*x(3)+110*x(1)*x(3)-
-1530*x(2)*x(3)+2625*x(3)^2
\end{verbatim}

On the other hand, our construction shows that there  are cases in which $\ell \star \ell'$ is a plane. As an example consider the points
\[
A_0=[1,2],\, A_1=[2,1],\, A_2=[1,3],\, A_3=[2,5]
\]  
giving
\[
P_0^\AAA=Q_0^\AAA=[1,1,1,1], \,P_1^\AAA=[3,\frac32, 4,\frac72], \,Q_1^\AAA=[\frac32, 3, \frac43, \frac75].
\]
For the lines $\ell^P=\overline{P_0^\AAA P_1^\AAA}$ and $\ell^Q=\overline{Q_0^\AAA Q_1^\AAA}$, we know, by  Theorem \ref{IC} that 
$\ell \star \ell'$ is a plane.
If we write down the equations of the two lines
\[
\ell^P:\begin{cases}
x_1+4x_2-5x_3=0\\
5x_0-2x_1-3x_2=0
\end{cases}
\]
\[
\ell^Q:\begin{cases}
x_1+24x_2-25x_3=0\\
10x_0-x_1-9x_2=0
\end{cases}.
\]
we can do a direct check in \texttt{Singular}:
\begin{verbatim}
> ring R=0,(x(0..3)),dp;
> ideal I1=x(1)+4*x(2)-5*x(3),5*x(0)-2*x(1)-3*x(2);
> ideal I2=x(1)+24*x(2)-25*x(3), 10*x(0)-x(1)-9*x(2);
> ideal K=HPr(I1,I2,3);
> K;
K[1]=40*x(0)-x(1)+36*x(2)-75*x(3)
\end{verbatim}
Hence we have two examples of coplanar lines with different behaviour of their Hadamard product. In these examples the lines are  are generated by respectively the following points
\[
\begin{array}{ccc}
S_1=[1,1,1,1], & S_2=[3,\frac32, 5,\frac72], & S_3=[\frac32, 3, \frac43, \frac75]\\
\\
P_0^\AAA=[1,1,1,1], &P_1^\AAA=[3,\frac32, 4,\frac72], & Q_1^\AAA=[\frac32, 3, \frac43, \frac75].
\end{array}\]
Notice that $S_1=P_0^\AAA$, and $S_3=Q_1^\AAA$ while $S_2$ and $P_1^\AAA$ differ only by an entry.

\end{Rmk}

\begin{Cor}
Let $Z_{a,b}^\AAA$ as in Theorem \ref{IC}, and let
\[
\begin{array}{l}
h=\sum_{t=0}^3 (-1)^{t+1} \alpha_t\beta_t|M(t+1)|x_t=0\\
\\
f=\sum_{
t=1}^3(-1)^t\alpha_t|N(1,t+1)|x_t=0\\
\\
g=\sum_{
t=1}^3(-1)^t\beta_t|N(1,t+1)|x_t=0.
\end{array}
\]
Then the ideal of $Z_{a,b}^\AAA$ is generated by $h, f^{\hada Q_{i_0}^\AAA}\cdots f^{\hada Q_{i_{b-1}}^\AAA},g^{\hada P_{i_0}^\AAA}\cdots g^{\hada P_{i_{a-1}}^\AAA}$.
\end{Cor}

\begin{proof}
Recall that $h$ and $f$ are the equation of $\ell^P$ and $h$ and $g$ are the equations of $\ell^Q$. By Theorem \ref{genset}, the equations of $Q_j^\AAA\hada \ell^P$ are given by $h^{\hada Q_j^\AAA}$ and $f^{\hada Q_j^\AAA}$ and the equations of $P_i^\AAA\hada \ell^Q$ are given by $h^{\hada P_i^\AAA}$ and $g^{\hada P_i^\AAA}$.
By Theorem \ref{IC}, since all the lines $Q_j^\AAA\hada \ell^P$ and $P_i^\AAA\hada \ell^Q$ are coplanar, then one of the generators for the ideal of each of them can be chosen to be the equation of the plane $H$ where they lie. Since this plane contains $P_0^\AAA\hada Q_0^\AAA=[1:1:1:1]$, $P_{i_1}^\AAA\hada Q_0^\AAA=P_{i_1}^\AAA$ and $P_0^\AAA\hada Q_{i_1}^\AAA=Q_{i_1}^\AAA$, we get that the equation of $H$ is exactly $h$.
Thus the ideal of $Q_j^\AAA\hada \ell^P$ is generated by $h$ and $f^{\hada Q_j^\AAA}$, while the ideal of $P_i^\AAA\hada \ell^Q$ is generated by $h$ and $g^{\hada P_i^\AAA}$, from which we get, by Theorem \ref{IC}, that $Z_{a,b}^\AAA$ is generated by $h, f^{\hada Q_{i_0}^\AAA}\cdots f^{\hada Q_{i_{b-1}}^\AAA},g^{\hada P_{i_0}^\AAA}\cdots g^{\hada P_{i_{a-1}}^\AAA}$.

\end{proof}



\section{Stick figures of lines via Hadamard product}\label{sticksec}
In this section we show how to get, via the Hadamard product, the stick figure of lines, in $\PP^3$, required for the construction in \cite{MN}.

To this aim, we consider, for a suitable choice of ${\mathcal{I}(a)}$ and ${\mathcal{I}(b)}$, the set $Z_{a,b}^\AAA$  defined in the previous section, and the line $L^\AAA$ defined in (\ref{eqline}) and we take their Hadamard product   $Z_{a,b}^\AAA\hada L^\AAA$. 

Before proving that $Z_{a,b}^\AAA\hada L^\AAA$ is a stick figure, we need two  preliminary lemmas.

\begin{Lem}\label{lemell}
 If $A_i\notin \mathcal{W}$, for $i=0,\dots, 3$, then $\ell^P\cap \Delta_0=\emptyset$ and $\ell^Q\cap \Delta_0=\emptyset$.
\end{Lem}
\begin{proof}
We prove the statement only for the $\ell^P$ since the proof is identical for  $\ell^Q$.
Suppose that, for example, $\ell^P$ intersects $\Delta_0$ in the point $E_0=[1:0:0:0]$. Notice that $P_j^\AAA\not=E_0$, for all $j$,  since $A_i\notin \mathcal{W}$ for $i=0,\dots, 3$. In particular, $P_1^\AAA$ has all coordinates different from zero. Since we are assuming that $E_0\in \ell^P$, $E_0$ can be written as a linear combination of $P_0^\AAA$ and $P_1^\AAA$, that is
\begin{equation}\label{E0}
[1:0:0:0]=\lambda [1:1:1:1]+\mu \left[1+\frac{\beta_0}{\alpha_0}:1+\frac{\beta_1}{\alpha_1}:1+\frac{\beta_2}{\alpha_2}:1+\frac{\beta_3}{\alpha_3}\right]
\end{equation}
which is possible only if
\[
\frac{\lambda}{\mu}=-\frac{\beta_1+\alpha_1}{\alpha_1} \mbox{ and } \frac{\alpha_1}{\beta_1}=\frac{\alpha_2}{\beta_2} =\frac{\alpha_3}{\beta_3}
\]
which is a contradiction since the points $A_i$ are distinct.
\end{proof}

\begin{Lem}\label{lemprijkl}
Let $r_{ijkl}$ be the line through $P_i^\AAA\hada Q_j^\AAA$ and $P_k^\AAA\hada Q_l^\AAA$. If $A_i\notin \mathcal{W}$, for $i=0,\dots, 3$, then $r_{ijkl}\cap \Delta_0=\emptyset$.
\end{Lem}

\begin{proof}
We distinguish three cases:
\begin{itemize}
\item[(1)] $i=k$ and $j\not=l$,
\item[(2)] $i\not=k$ and $j=l$,
\item[(3)] $i\not=k$ and $j\not=l$.
\end{itemize}
If we are in case (1), the line through $P_i^\AAA\hada Q_j^\AAA$  and $P_i^\AAA\hada Q_l^\AAA$ is the line $P_i^\AAA\hada \ell^Q$ which does not intersects $\Delta_0$ by Corollary \ref{corgenset} and Lemma \ref{lemell}.
Similarly, if we are in case (2), the line through $P_i^\AAA\hada Q_j^\AAA$ and $P_k^\AAA\hada Q_j^\AAA$ is the line $Q_j^\AAA\hada \ell^P$ which, again, does not intersects $\Delta_0$ by Corollary \ref{corgenset} and Lemma \ref{lemell}.

For case (3), suppose that $r_{ijkl}$ intersects $\Delta_0$ in $E_0=[1:0:0:0]$ (the other cases being similar). Notice that $P_i^\AAA\hada Q_j^\AAA\not=E_0$, and $P_k^\AAA\hada Q_l^\AAA\not=E_0$  since $A_i\notin \mathcal{W}$ for $i=0,\dots, 3$. Since we are assuming that $E_0\in r_{ijkl}$, $E_0$ can be written as a linear combination of $P_i^\AAA\hada Q_j^\AAA$ and $P_k^\AAA\hada Q_l^\AAA$
{\small{
\begin{equation}\label{E0ijkr}
\begin{array}{l}
[1:0:0:0]=\\
\\
\lambda \left[ \frac{(\alpha_0+i\beta_0)(j\alpha_0+\beta_0)}{\alpha_0\beta_0}: \frac{(\alpha_1+i\beta_1)(j\alpha_1+\beta_1)}{\alpha_1\beta_1}:\frac{(\alpha_2+i\beta_2)(j\alpha_2+\beta_2)}{\alpha_2\beta_2}: \frac{(\alpha_3+i\beta_3)(j\alpha_3+\beta_3)}{\alpha_3\beta_3}\right]+\\
\\
\mu\left[ \frac{(\alpha_0+k\beta_0)(l\alpha_0+\beta_0)}{\alpha_0\beta_0}: \frac{(\alpha_1+k\beta_1)(l\alpha_1+\beta_1)}{\alpha_1\beta_1}:\frac{(\alpha_2+k\beta_2)(l\alpha_2+\beta_2)}{\alpha_2\beta_2}: \frac{(\alpha_3+k\beta_3)(l\alpha_3+\beta_3)}{\alpha_3\beta_3}\right]
\end{array}
\end{equation}}}
and looking at all the coordinates but the first, this means
\[
\lambda  \frac{(\alpha_s+i\beta_s)(j\alpha_s+\beta_s)}{\alpha_s\beta_s}=-\mu\frac{(\alpha_s+k\beta_s)(l\alpha_s+\beta_s)}{\alpha_s\beta_s}
\mbox{ for } s=1,2,3
\]
or equivalently 
\begin{equation}\label{fineqijkl}
-\frac{\mu}{\lambda}= \frac{(\alpha_s+i\beta_s)(j\alpha_s+\beta_s)}{(\alpha_s+k\beta_s)(l\alpha_s+\beta_s)} \mbox{ for } s=1,2,3
\end{equation}
which gives rise to the same set of equations (\ref{eqbr}) of Theorem \ref{IC}.
Arguing as in the proof of Theorem \ref{IC}, but considering that now we have one less equation (since we are not considering the first coordinate), we get that any non-zero solution of (\ref{fineqijkl}) requires that there is a pair $(i_1,i_2)$ of indices such that $\frac{\alpha_{i_1}}{\beta_{i_1}}=\frac{\alpha_{i_2}}{\beta_{i_2}}$ which is a contradiction since the points $A_i$ are distinct.
 \end{proof}

Coming back to $Z_{a,b}^\AAA\hada L^\AAA$, we first prove that no pairs of points $P_i^\AAA\hada Q_j^\AAA,P_k^\AAA\hada Q_l^\AAA\in Z_{a,b}^\AAA$ can give $P_i^\AAA\hada Q_j^\AAA\hada L^\AAA=P_k^\AAA\hada Q_l^\AAA\hada L^\AAA$.
\begin{Prop}\label{cardinalitylines}
In the same hypothesis of Theorem \ref{IC}, $Z_{a,b}^\AAA\hada L^\AAA$ is a set of $ab$ distinct lines, for any choice of positive integers $a$ and $b$ and sets ${\mathcal{I}}(a)$ and ${\mathcal{I}}(b)$.
\end{Prop}

\begin{proof}
Since, by hypothesis, $P_i^\AAA\hada Q_j^\AAA\notin \Delta_2$, by Lemma \ref{LempL}, one has that $P_i^\AAA\hada Q_j^\AAA\hada L^\AAA$ is a line for all  $i$ and $j$ with  $i \in {\mathcal{I}}(a)$ and $j \in {\mathcal{I}}(b)$.  Let us show now that if $P_i^\AAA\hada Q_j^\AAA\not=P_k^\AAA\hada Q_l^\AAA$ then $P_i^\AAA\hada Q_j^\AAA\hada L^\AAA\not=P_k^\AAA\hada Q_l^\AAA\hada L^\AAA$.
We distinguish three cases.

If $i=k$ and $j\not=l$ then $P_i^\AAA\hada Q_j^\AAA \not=P_i^\AAA\hada Q_l^\AAA$. By Lemma \ref{lemell}, $\ell^Q\cap \Delta_0=\emptyset$ which implies $P_i^\AAA\hada \ell^Q\cap\Delta_0=\emptyset$. Since
\begin{itemize}
\item[i)] $P_i^\AAA\hada Q_j^\AAA, P_i^\AAA\hada Q_l^\AAA\notin \Delta_2$,
\item[ii)] $L^\AAA\notin\Delta_1$,
\item[iii)] $\langle P_i^\AAA\hada Q_j^\AAA, P_i^\AAA\hada Q_l^\AAA\rangle=P_i^\AAA\hada \ell^Q$,
\end{itemize}
we can apply Lemma \ref{LempqL}, obtaining $P_i^\AAA\hada Q_j^\AAA\hada L^\AAA\not=P_i^\AAA\hada Q_l^\AAA\hada L^\AAA$.

The  case $i\not=k$ and $j=l$ is similar to the previous one. The same proof, but using the line  $\ell^P$, gives $P_i^\AAA\hada Q_j^\AAA\hada L^\AAA\not=P_k^\AAA\hada Q_j^\AAA\hada L^\AAA$.

Finally if $i\not=k$ and $j\not=l$ then $P_i^\AAA\hada Q_j^\AAA \not=P_k^\AAA\hada Q_l^\AAA$. By Lemma \ref{lemprijkl}, $r_{ijkl}\cap \Delta_0=\emptyset$. Since
\begin{itemize}
\item[i)] $P_i^\AAA\hada Q_j^\AAA, P_k^\AAA\hada Q_l^\AAA\notin \Delta_2$,
\item[ii)] $L^\AAA\notin\Delta_1$,
\item[iii)] $\langle P_i^\AAA\hada Q_j^\AAA, P_k^\AAA\hada Q_l^\AAA\rangle=r_{ijkl}$.
\end{itemize}
we can again apply Lemma \ref{LempqL}, obtaining $P_i^\AAA\hada Q_j^\AAA\hada L^\AAA\not=P_k^\AAA\hada Q_l^\AAA\hada L^\AAA$.

Thus we conclude that $Z_{a,b}^\AAA\hada L^\AAA$ consists of $ab$ distinct lines.
\end{proof}

We study now the intersection properties of the set $Z_{a,b}^\AAA\star L^\AAA$. More precisely we have the following.

\begin{Prop}\label{interstick} Assume that $1\notin  {\mathcal{I}}(a)\cup  {\mathcal{I}}(b)$. In the same hypothesis of Theorem \ref{IC}, let $P_i^\AAA\star Q_j^\AAA$ and $P_k^\AAA\star Q_l^\AAA$ in $Z_{a,b}^\AAA$. Then $P_i^\AAA\star Q_j^\AAA \star L^\AAA\cap P_k^\AAA\star Q_l^\AAA\star L^\AAA\not=\emptyset$ if and only if $i=k$ or $j=l$.
Moreover, 
\begin{itemize}
\item[i)] if $j\not= l$, the intersection $P_i^\AAA\star Q_j^\AAA \star L^\AAA\cap P_i^\AAA\star Q_l^\AAA\star L^\AAA$ is given by

{\large{\begin{equation}\label{punto1}
\left[\begin{array}{c}
-\frac{ ( \alpha_0+i\beta_0)(j  \alpha_0+\beta_0)(l \alpha_0+\beta_0)}{\alpha_0\beta_0(\alpha_0\beta_1-\alpha_1\beta_0)(\alpha_0\beta_2-\alpha_2\beta_0)(\alpha_0\beta_3-\alpha_3\beta_0)}\\
\\
\frac{ ( \alpha_1+i\beta_1)(j  \alpha_1+\beta_1)(l \alpha_1+\beta_1)}{\alpha_1\beta_1(\alpha_0\beta_1-\alpha_1\beta_0)(\alpha_1\beta_2-\alpha_2\beta_1)(\alpha_1\beta_3-\alpha_3\beta_1)}\\
\\
-\frac{( \alpha_2+i\beta_2)(j  \alpha_2+\beta_2)(l \alpha_2+\beta_2) }{\alpha_2\beta_2(\alpha_0\beta_2-\alpha_2\beta_0)(\alpha_1\beta_2-\alpha_2\beta_1)(\alpha_2\beta_3-\alpha_3\beta_2)}\\
\\
\frac{( \alpha_3+i\beta_3)(j  \alpha_3+\beta_3) (l \alpha_3+\beta_3)}{\alpha_3\beta_3(\alpha_0\beta_3-\alpha_3\beta_0)(\alpha_1\beta_3-\alpha_3\beta_1)(\alpha_2\beta_3-\alpha_3\beta_2)}
\end{array}\right];
\end{equation}}}
\item[ii)] if $i\not= k$, the intersection $P_i^\AAA\star Q_j^\AAA \star L^\AAA\cap P_k^\AAA\star Q_j^\AAA\star L^\AAA$ is given by
{\large{\begin{equation}\label{punto2}
\left[\begin{array}{c}
-\frac{ ( \alpha_0+i\beta_0)( \alpha_0+k\beta_0)(j  \alpha_0+\beta_0)}{\alpha_0\beta_0(\alpha_0\beta_1-\alpha_1\beta_0)(\alpha_0\beta_2-\alpha_2\beta_0)(\alpha_0\beta_3-\alpha_3\beta_0)}\\
\\
\frac{ ( \alpha_1+i\beta_1)( \alpha_1+k\beta_1)(j  \alpha_1+\beta_1)}{\alpha_1\beta_1(\alpha_0\beta_1-\alpha_1\beta_0)(\alpha_1\beta_2-\alpha_2\beta_1)(\alpha_1\beta_3-\alpha_3\beta_1)}\\
\\
-\frac{ ( \alpha_2+i\beta_2)( \alpha_2+k\beta_2)(j  \alpha_2+\beta_2) }{\alpha_2\beta_2(\alpha_0\beta_2-\alpha_2\beta_0)(\alpha_1\beta_2-\alpha_2\beta_1)(\alpha_2\beta_3-\alpha_3\beta_2)}\\
\\
\frac{ ( \alpha_3+i\beta_3)( \alpha_3+k\beta_3)(j  \alpha_3+\beta_3)}{\alpha_3\beta_3(\alpha_0\beta_3-\alpha_3\beta_0)(\alpha_1\beta_3-\alpha_3\beta_1)(\alpha_2\beta_3-\alpha_3\beta_2)}
\end{array}\right].
\end{equation}}}
\end{itemize}
\end{Prop}

\begin{proof}
By Theorem \ref{genset} one has that the equations of $P_i^\AAA\star Q_j^\AAA \star L^\AAA$ are
\[
\begin{cases}
(\alpha_0x_0+\alpha_1x_1+\alpha_2x_2+\alpha_3x_3)^{\star (P_i^\AAA\star Q_j^\AAA)} =0\\
(\beta_0x_0+\beta_1x_1+\beta_2x_2+\beta_3x_3)^{\star (P_i^\AAA\star Q_j^\AAA)}=0
\end{cases}
\]
which can be written explicitly as
{\tiny{\[
\begin{cases}
\frac{\alpha_0^2   \beta_0}{(\alpha_0 + i \beta_0) (j \alpha_0 + \beta_0)} x_0+
\frac{\alpha_1^2   \beta_1}{(\alpha_1 + i \beta_1) (j \alpha_1 + \beta_1)} x_1+
\frac{\alpha_2^2   \beta_2}{(\alpha_2 + i \beta_2) (j \alpha_2 + \beta_2)} x_2+
\frac{\alpha_3^2   \beta_3}{(\alpha_3 + i \beta_3) (j \alpha_3 + \beta_3)} x_3=0\\

\frac{\alpha_0   \beta_0^2}{(\alpha_0 + i \beta_0) (j \alpha_0 + \beta_0)} x_0+
\frac{\alpha_1   \beta_1^2}{(\alpha_1 + i \beta_1) (j \alpha_1 + \beta_1)} x_1+
\frac{\alpha_2   \beta_2^2}{(\alpha_2 + i \beta_2) (j \alpha_2 + \beta_2)} x_2+
\frac{\alpha_3   \beta_3^2}{(\alpha_3 + i \beta_3) (j \alpha_3 + \beta_3)} x_3=0
\end{cases}.
\]}}
Similarly the equations of $P_k^\AAA\star Q_l^\AAA \star L^\AAA$ are
\[
\begin{cases}
(\alpha_0x_0+\alpha_1x_1+\alpha_2x_2+\alpha_3x_3)^{\star (P_k^\AAA\star Q_l^\AAA)} =0\\
(\beta_0x_0+\beta_1x_1+\beta_2x_2+\beta_3x_3)^{\star (P_k^\AAA\star Q_l^\AAA)}=0
\end{cases}
\]
which can be written explicitly as
{\tiny{\[
\begin{cases}
\frac{\alpha_0^2   \beta_0}{(\alpha_0 + k \beta_0) (l \alpha_0 + \beta_0)} x_0+
\frac{\alpha_1^2   \beta_1}{(\alpha_1 + k \beta_1) (l \alpha_1 + \beta_1)} x_1+
\frac{\alpha_2^2   \beta_2}{(\alpha_2 + k \beta_2) (l \alpha_2 + \beta_2)} x_2+
\frac{\alpha_3^2   \beta_3}{(\alpha_3 + k \beta_3) (l \alpha_3 + \beta_3)} x_3=0\\

\frac{\alpha_0   \beta_0^2}{(\alpha_0 + k \beta_0) (l \alpha_0 + \beta_0)} x_0+
\frac{\alpha_1   \beta_1^2}{(\alpha_1 + k \beta_1) (l \alpha_1 + \beta_1)} x_1+
\frac{\alpha_2   \beta_2^2}{(\alpha_2 + k \beta_2) (l \alpha_2 + \beta_2)} x_2+
\frac{\alpha_3   \beta_3^2}{(\alpha_3 + k \beta_3) (l \alpha_3 + \beta_3)} x_3=0
\end{cases}.
\]}}
Passing to the system of the two lines $P_i^\AAA\star Q_j^\AAA \star L^\AAA$ and $P_k^\AAA\star Q_l^\AAA\star L^\AAA$ 
 {\tiny{\begin{equation}\label{sistemarette}
\begin{cases}
\frac{\alpha_0^2   \beta_0}{(\alpha_0 + i \beta_0) (j \alpha_0 + \beta_0)} x_0+
\frac{\alpha_1^2   \beta_1}{(\alpha_1 + i \beta_1) (j \alpha_1 + \beta_1)} x_1+
\frac{\alpha_2^2   \beta_2}{(\alpha_2 + i \beta_2) (j \alpha_2 + \beta_2)} x_2+
\frac{\alpha_3^2   \beta_3}{(\alpha_3 + i \beta_3) (j \alpha_3 + \beta_3)} x_3=0\\

\frac{\alpha_0   \beta_0^2}{(\alpha_0 + i \beta_0) (j \alpha_0 + \beta_0)} x_0+
\frac{\alpha_1   \beta_1^2}{(\alpha_1 + i \beta_1) (j \alpha_1 + \beta_1)} x_1+
\frac{\alpha_2   \beta_2^2}{(\alpha_2 + i \beta_2) (j \alpha_2 + \beta_2)} x_2+
\frac{\alpha_3   \beta_3^2}{(\alpha_3 + i \beta_3) (j \alpha_3 + \beta_3)} x_3=0\\

\frac{\alpha_0^2   \beta_0}{(\alpha_0 + k \beta_0) (l \alpha_0 + \beta_0)} x_0+
\frac{\alpha_1^2   \beta_1}{(\alpha_1 + k \beta_1) (l \alpha_1 + \beta_1)} x_1+
\frac{\alpha_2^2   \beta_2}{(\alpha_2 + k \beta_2) (l \alpha_2 + \beta_2)} x_2+
\frac{\alpha_3^2   \beta_3}{(\alpha_3 + k \beta_3) (l \alpha_3 + \beta_3)} x_3=0\\

\frac{\alpha_0   \beta_0^2}{(\alpha_0 + k \beta_0) (l \alpha_0 + \beta_0)} x_0+
\frac{\alpha_1   \beta_1^2}{(\alpha_1 + k \beta_1) (l \alpha_1 + \beta_1)} x_1+
\frac{\alpha_2   \beta_2^2}{(\alpha_2 + k \beta_2) (l \alpha_2 + \beta_2)} x_2+
\frac{\alpha_3   \beta_3^2}{(\alpha_3 + k \beta_3) (l \alpha_3 + \beta_3)} x_3=0
\end{cases}
\end{equation}}}
one has the following matrix of coefficients
\[
{\mathcal{M}}=\begin{pmatrix}
\frac{\alpha_0^2   \beta_0}{(\alpha_0 + i \beta_0) (j \alpha_0 + \beta_0)}&
\frac{\alpha_1^2   \beta_1}{(\alpha_1 + i \beta_1) (j \alpha_1 + \beta_1)}&
\frac{\alpha_2^2   \beta_2}{(\alpha_2 + i \beta_2) (j \alpha_2 + \beta_2)} &
\frac{\alpha_3^2   \beta_3}{(\alpha_3 + i \beta_3) (j \alpha_3 + \beta_3)} \cr

\frac{\alpha_0   \beta_0^2}{(\alpha_0 + i \beta_0) (j \alpha_0 + \beta_0)} &
\frac{\alpha_1   \beta_1^2}{(\alpha_1 + i \beta_1) (j \alpha_1 + \beta_1)} &
\frac{\alpha_2   \beta_2^2}{(\alpha_2 + i \beta_2) (j \alpha_2 + \beta_2)} &
\frac{\alpha_3   \beta_3^2}{(\alpha_3 + i \beta_3) (j \alpha_3 + \beta_3)} \cr

\frac{\alpha_0^2   \beta_0}{(\alpha_0 + k \beta_0) (l \alpha_0 + \beta_0)} &
\frac{\alpha_1^2   \beta_1}{(\alpha_1 + k \beta_1) (l \alpha_1 + \beta_1)} &
\frac{\alpha_2^2   \beta_2}{(\alpha_2 + k \beta_2) (l \alpha_2 + \beta_2)} &
\frac{\alpha_3^2   \beta_3}{(\alpha_3 + k \beta_3) (l \alpha_3 + \beta_3)} \cr

\frac{\alpha_0   \beta_0^2}{(\alpha_0 + k \beta_0) (l \alpha_0 + \beta_0)} &
\frac{\alpha_1   \beta_1^2}{(\alpha_1 + k \beta_1) (l \alpha_1 + \beta_1)} &
\frac{\alpha_2   \beta_2^2}{(\alpha_2 + k \beta_2) (l \alpha_2 + \beta_2)} &
\frac{\alpha_3   \beta_3^2}{(\alpha_3 + k \beta_3) (l \alpha_3 + \beta_3)} \cr

\end{pmatrix}.
\]
Clearly this matrix has rank  greater than or equal to 3, otherwise the two lines $P_i^\AAA\star Q_j^\AAA \star L^\AAA$ and $P_k^\AAA\star Q_l^\AAA\star L^\AAA$  will be coincident, in contradiction with Proposition \ref{cardinalitylines}.

Computing the determinant of ${\mathcal{M}}$ one has
\begin{equation*}
\det({\mathcal{M}})=\frac{\left(\prod_{t=0}^3a_tb_t\right)\left(\prod_{0\leq s<r\leq 3}(a_sb_r-a_rb_s)\right) (i-k)(j-l)(jk-1)(il-1)} {\prod_{t=0}^3\big((a_t+ib_t)(a_t+kb_t)(ja_t+b_t)(la_t+b_t)\big)}.
\end{equation*}
By definition of the points $A_i$ and by Remark \ref{rmkLD} we know that the two terms  $\left(\prod_{t=0}^3a_tb_t\right)$ and $\left(\prod_{0\leq s<r\leq 3}(a_sb_r-a_rb_s)\right)$ are different from 0.
By the condition $1\notin  {\mathcal{I}}(a)\cup  {\mathcal{I}}(b)$ one as that $(jk-1)(il-1)\not=0$ for all $i,k \in  {\mathcal{I}}(a)$ and all $j,l \in {\mathcal{I}}(b)$.
Hence ${\mathcal{M}}$  has rank 4 when $i\not=k$ and $j\not=l$ and has rank 3 when $i=k$ or $j=l$, which concludes the first part of the proof.
The second part of the proof follows directly substituting the values in (\ref{punto1}) in the system (\ref{sistemarette}) taking $i=k$, and the values in (\ref{punto2}) in the same system (\ref{sistemarette}) but taking $j=l$.
\end{proof}

\begin{Rmk} \rm
Although, by  Proposition \ref{cardinalitylines}, any choice of the sets ${\mathcal{I}}(a)$ and ${\mathcal{I}}(b)$ always gives a set of $ab$ distinct lines, the condition $1\notin  {\mathcal{I}}(a)\cup  {\mathcal{I}}(b)$ of the previous proposition is mandatory to avoid extra intersections among the lines in $Z_{a,b}^\AAA\hada L^\AAA$. In fact, without this condition $Z_{a,b}^\AAA\hada L^\AAA$ could be still a stick figure, but it is not complete intersection.
\end{Rmk}

As a corollary we get the following fact.
\begin{Cor}\label{CORSTICK} Assume that $1\notin  {\mathcal{I}}(a)\cup  {\mathcal{I}}(b)$. With the same hypothesis of Theorem \ref{IC}, one has:
\begin{itemize} 
\item $P^\AAA_{i_0}\star Q^\AAA_{j} \star L^\AAA, \dots, P^\AAA_{i_{a-1}}\star Q^\AAA_j \star L^\AAA$ are coplanar for all $j\in  {\mathcal{I}}(b)$;
\item $P^\AAA_i \star Q^\AAA_{i_0}\star L^\AAA, \dots, P^\AAA_i \star Q^\AAA_{i_{b-1}} \star L^\AAA$ are coplanar for all $i\in  {\mathcal{I}}(a)$.
\end{itemize}
\end{Cor}

We have now all ingredients to state the main result of this section.

\begin{Thm}\label{Thmstick}
Assume that $1\notin  {\mathcal{I}}(a)\cup  {\mathcal{I}}(b)$. With the same hypothesis of Theorem \ref{IC}, $Z_{a,b}^\AAA\star L^\AAA$ is a stick figure of  $ab$ lines in $\PP^3$. Moreover  $Z_{a,b}^\AAA\star L^\AAA$ is a complete intersection.
\end{Thm}

\begin{proof}
By Proposition \ref{cardinalitylines}, we know that $Z_{a,b}^\AAA\star L^\AAA$ consists of $ab$ distinct lines.
By Corollary \ref{CORSTICK}, it follows that $Z_{a,b}^\AAA\star L^\AAA$  is a complete intersection.
By the first part of Proposition \ref{interstick}, we know that two lines in $Z_{a,b}^\AAA\star ^\AAA$ intersect in a space of dimension at most 0. 
By the second part of Proposition \ref{interstick}, we know  that the coordinates of the point of intersection of two lines $P_i^\AAA\star Q_j^\AAA \star L^\AAA$ and $P_i^\AAA\star Q_l^\AAA\star L^\AAA$  (resp. $P_i^\AAA\star Q_j^\AAA \star L^\AAA$ and $P_k^\AAA\star Q_j^\AAA\star L^\AAA$ ) are dependent of the indices $i,j$ and $l$ (resp. $i,j$ and $k$) assuring us that  three lines in $Z_{a,b}^\AAA\star L^\AAA$ intersect in a space of dimension at most -1.
Hence $Z_{a,b}^\AAA\star L^\AAA$ satisfies the conditions to be a stick figure of lines and the statement is proved. 

\end{proof}


\section{Gorenstein sets of points}\label{gorsec}
As a final step of our construction, we apply  the procedure described in Section \ref{MNsec} to our stick figure to get a Gorenstein set of points in $\PP^3$ with a given $h-$vector.

Again, let
$${\mathbf h}=(h_0 , h_1 , \dots , h_{s})=(1,3,h_2,\dots,h_{t-1},h_t,h_t,\dots,h_t,h_{t-1},\dots,h_2,3,1)$$
be a SI-sequence, and consider the first difference
$$\Delta{\mathbf h}=(1,2,h_2-h_1,\dots,h_t-h_{t-1},0,0,\dots,0,h_{t-1}-h_t,\dots,-2,-1).
$$  

Define the two sequences ${\mathbf a}=(a_0,\dots,a_t)$ and ${\mathbf g}=(g_0,\dots,g_{s+1})$ as expressed in (\ref{costruzionea}) and (\ref{costruzioneg}).
As already said in Section \ref{MNsec}, ${\mathbf g}$ is the $h$-vector of a complete intersection, $X$, of
two surfaces in $\PP^3$ of degree $t+1$ and $s-t+2$. 

Hence we consider, as $X$, the stick figure $Z_{t+1,s-t+2}^\AAA\star L^\AAA$ (for a suitable choice of $\AAA$, ${\mathcal{I}}(t+1)$ and ${\mathcal{I}}(s-t+2)$ with the hypotheses of Theorems \ref{IC} and \ref{Thmstick}).

If we set 
\begin{equation}\label{IaIb}
{\mathcal{I}}(t+1)=\{u_0, \dots, u_{t}\} \mbox{ and } {\mathcal{I}}(s-t+2)=\{v_0, \dots, v_{s-t+1}\}, 
\end{equation}
then  the aCM scheme $C_1$ with $h-$vector ${\mathbf a}$ is given by the following set of lines in $Z_{t+1,s-t+2}^\AAA\star L^\AAA$:
\[
P^\AAA_{u_i}\star Q_{v_j}^\AAA\star L^\AAA \mbox{ for } j=0, \dots, a_i-1 \mbox{ and } i =0, \dots, t.
\] 
and, obviously, the residual scheme $C_2$ is the set of lines in $Z_{t+1,s-t+2}^\AAA\star L^\AAA$ and not in $C_1$.

We use the following notation for the points of intersections of lines in the stick figure $Z_{t+1,s-t+2}^\AAA\star L^\AAA$:
\[
G_{i\{j,k\}}=P_{u_i}^\AAA\star Q_{v_j}^\AAA\star L^\AAA \cap P_{u_i}^\AAA\star Q_{v_k}^\AAA\star L^\AAA
\]
and
\[
G_{\{i,k\}j}=P_{u_i}^\AAA\star Q_{v_j}^\AAA\star L^\AAA \cap P_{u_k}^\AAA\star Q_{v_j}^\AAA\star L^\AAA.
\]
\begin{Thm} Let ${\mathbf h}, {\mathbf a}$ and ${\mathbf g}$ be as above. Then the set of points
\[
\begin{cases}
 G_{i\{j,k\}}  & \mbox{ with } 0\le j\le a_{i}-1  \mbox{ and } a_{i}\le k\le s-t+1 \mbox{ for } i=0,\dots, t\\
\\
G_{\{i,k\}j} & \mbox{ with } \min\{a_i,a_k\}\leq j \leq \max\{a_i,a_k\}-1 \mbox{ for } 0\leq i < k \leq t
\end{cases}
\]
is a Gorenstein zeroscheme with $h-$vector ${\mathbf h}$.
\end{Thm}

\begin{proof} This follows directly from  Theorem \ref{LiaisonThm}   and Theorem \ref{hvettore}.
\end{proof}

Using the description, in Proposition \ref{interstick}, of intersections in the stick figure, we can state the previous theorem in terms of the coordinates of the points in the desired Gorenstein set.

Consider ${\mathcal{I}}(t+1)$ and ${\mathcal{I}}(s-t+2)$ as in (\ref{IaIb}). Denote by $[V_{i\{j,k\}}]$ the point whose coordinates are

\begin{equation}\label{puntoo}
\begin{split}
[V_{i\{j,k\}}]_0=&-\frac{ ( \alpha_0+u_i\beta_0)(v_j  \alpha_0+\beta_0)(v_k \alpha_0+\beta_0)}{\alpha_0\beta_0(\alpha_0\beta_1-\alpha_1\beta_0)(\alpha_0\beta_2-\alpha_2\beta_0)(\alpha_0\beta_3-\alpha_3\beta_0)}\\
[V_{i\{j,k\}}]_1=&\frac{ ( \alpha_1+u_i\beta_1)(v_j  \alpha_1+\beta_1)(v_k \alpha_1+\beta_1)}{\alpha_1\beta_1(\alpha_0\beta_1-\alpha_1\beta_0)(\alpha_1\beta_2-\alpha_2\beta_1)(\alpha_1\beta_3-\alpha_3\beta_1)}\\
[V_{i\{j,k\}}]_2 = & -\frac{( \alpha_2+u_i\beta_2)(v_jj  \alpha_2+\beta_2)(v_k \alpha_2+\beta_2) }{\alpha_2\beta_2(\alpha_0\beta_2-\alpha_2\beta_0)(\alpha_1\beta_2-\alpha_2\beta_1)(\alpha_2\beta_3-\alpha_3\beta_2)}\\
[V_{i\{j,k\}}]_3= & \ \frac{( \alpha_3+u_i\beta_3)(v_j  \alpha_3+\beta_3) (v_k \alpha_3+\beta_3)}{\alpha_3\beta_3(\alpha_0\beta_3-\alpha_3\beta_0)(\alpha_1\beta_3-\alpha_3\beta_1)(\alpha_2\beta_3-\alpha_3\beta_2)}
\end{split}
\end{equation}
and by $[V_{\{i,k\}j}]$ the point whose coordinates are
\begin{equation}\label{puntov}
\begin{split}
[V_{\{i,k\}j}]_0=&-\frac{ ( \alpha_0+u_i\beta_0)( \alpha_0+u_k\beta_0)(v_j  \alpha_0+\beta_0)}{\alpha_0\beta_0(\alpha_0\beta_1-\alpha_1\beta_0)(\alpha_0\beta_2-\alpha_2\beta_0)(\alpha_0\beta_3-\alpha_3\beta_0)}\\
[V_{\{i,k\}j}]_1=&\frac{ ( \alpha_1+u_i\beta_1)( \alpha_1+u_k\beta_1)(v_j  \alpha_1+\beta_1)}{\alpha_1\beta_1(\alpha_0\beta_1-\alpha_1\beta_0)(\alpha_1\beta_2-\alpha_2\beta_1)(\alpha_1\beta_3-\alpha_3\beta_1)}\\
[V_{\{i,k\}j}]_2=&-\frac{ ( \alpha_2+u_i\beta_2)( \alpha_2+u_k\beta_2)(v_j  \alpha_2+\beta_2) }{\alpha_2\beta_2(\alpha_0\beta_2-\alpha_2\beta_0)(\alpha_1\beta_2-\alpha_2\beta_1)(\alpha_2\beta_3-\alpha_3\beta_2)}\\
[V_{\{i,k\}j}]_3=&\frac{ ( \alpha_3+u_i\beta_3)( \alpha_3+u_k\beta_3)(v_j  \alpha_3+\beta_3)}{\alpha_3\beta_3(\alpha_0\beta_3-\alpha_3\beta_0)(\alpha_1\beta_3-\alpha_3\beta_1)(\alpha_2\beta_3-\alpha_3\beta_2)}.
\end{split}
\end{equation}

\begin{Cor}\label{corfinale} 
 Let  ${\mathbf h}$ be an admissible $h-$vector for a Gorenstein zeroscheme in $\PP^3$ of the form ${\mathbf h}=(h_0  , \dots , h_{s})=(1,3,h_2,\dots,h_{t-1}, h_t,h_t,\dots,h_t,h_{t-1},\dots,3,1)$ and let $a_i=h_i-h_{i-1}$ for $0\leq{i}\leq{t}$. Fix four distinct points $A_i=[\alpha_i:\beta_i]$ in $\PP^1\setminus (\Delta_0\cup \mathcal{W})$, for $i=0,\dots, 3$ and fix the sets of nonnegative integers ${\mathcal{I}}(t+1)=\{u_0, \dots, u_{t}\}$  and ${\mathcal{I}}(s-t+2)=\{v_0, \dots, v_{s-t+1}\}$ with $0\in {\mathcal{I}}(t+1)\cap {\mathcal{I}}(s-t+2)$ and $1\notin {\mathcal{I}}(t+1)\cup {\mathcal{I}}(s-t+2)$.  Then the set of points

\[
\begin{cases}
[V_{i\{j,k\}}] &  \mbox{ with }  0\le j\le a_{i}-1, a_{i}\le k\le s-t+1, \mbox{ for } i=0,\dots, t\\
\\
[V_{\{i,k\}j}] &   \mbox{ with } \min\{a_i,a_k\}\leq j \leq \max\{a_i,a_k\}-1,  \mbox{ for } 0\leq i < k \leq t
\end{cases}
\]
is a Gorenstein zeroscheme with $h-$vector ${\mathbf h}$.
\end{Cor}

\begin{Ex} \rm
Let  ${\mathbf h}$ be $h$-vector $(1,3,4,3,1)$ of Example \ref{punti}. One has $t=2$, $s=4$ and ${\mathbf a}=(1,2,1)$.

Fix
\[
A_0=[1:1], \, A_1=[1:2],  \, A_2=[1:3], \,  A_3=[1:4]
\]
and
\[
\begin{array}{rl}
{\mathcal{I}}(t+1)=\{u_0, u_1, u_{2}\}&=\{0,2,4\}\\
\\
{\mathcal{I}}(s-t+2)=\{v_0, v_1, v_2, v_3\}&=\{0,2,4,6\}.
\end{array}
\]
Substituting these values in (\ref{puntoo}) and (\ref{puntov}) we get, by Corollary  \ref{corfinale}, that the Gorenstein set of points with $h$-vector $(1,3,4,3,1)$ is given by

\noindent
$
[V_{i\{j,k\}}]=
\left\lbrack
\begin{array}{r}
-\frac{(2k+1)(2j+1)(2i+1)}{6}\\
\frac{(2k+2)(2j+2)(4i+1)}{4}\\
-\frac{(2k+3)(2j+3)(6i+1)}{6}\\
\frac{(2k+4)(2j+4)(8i+1)}{24}
\end{array}\right\rbrack \mbox{ with }  0\le j\le a_{i}-1, a_{i}\le k\le 3, \mbox{ for } i=0,1, 2
$
\vskip0.2cm
\noindent
and
\vskip0.2cm
\noindent
$
[V_{\{i,k\}j}] =
\left\lbrack
\begin{array}{r}
-\frac{(2k+1)(2j+1)(2i+1)}{6}\\
\frac{(4k+1)(2j+2)(4i+1)}{4}\\
-\frac{(6k+1)(2j+3)(6i+1)}{6}\\
\frac{(8k+1)(2j+4)(8i+1)}{24}
\end{array}\right\rbrack \begin{array}{c} \mbox{ with }  \min\{a_i,a_k\}\leq j \leq \max\{a_i,a_k\}-1, \\ \\ \mbox{ for } 0\leq i < k \leq 2.\end{array}
$
\vskip0.2cm
\noindent
that is
\[
\begin{array}{cccccc}
\lbrack-\frac{1}{2}:2:-\frac{5}{2}:1\rbrack, &  & \lbrack-\frac56:3:-\frac72:\frac43\rbrack, & & \lbrack-\frac76:4:-\frac92:\frac53\rbrack,\\
\\
\lbrack-\frac{5}{2}:15:-\frac{49}{2}:12\rbrack, & & \lbrack-\frac72:20:-\frac{63}{2}:15 \rbrack,  && \lbrack, -\frac{15}{2}:30:-\frac{245}{6}:18 \rbrack, \\
\\
 \lbrack-\frac{21}{2}:40:-\frac{105}{2}:\frac{45}{2} \rbrack, &&  \lbrack -\frac{5}{2}: 18: -\frac{65}{2}: 17 \rbrack, &&  \lbrack -\frac{25}{6}:27:-\frac{91}{2}:\frac{68}{3}\rbrack,\\
 \\
   \lbrack -\frac{35}{6}:36:-\frac{117}{2}:\frac{85}{3}\rbrack, &&    \lbrack -\frac32:5:-\frac{35}{6}:\frac94 \rbrack, &&    \lbrack-\frac{15}{2}:45:-\frac{455}{6}:\frac{153}{4} \rbrack .   
\end{array}
\]
We can check in \texttt{Singular} if this set of points is Gorenstein.
The procedure \texttt{IP(n,M)} computes the ideal of a set of points given in matrix form $M$, where each column of $M$ represents a point. The procedure \texttt{HF(n,I,t)} computes the Hilbert function of an ideal $I$, in degree $t$. The integer $n$ refers to the number of variables in the polynomial ring $k[x_0,\dots, x_n]$.

\begin{verbatim}

int n=3;
ring R=0,(x(0..n)),dp;

matrix G[4][12]=-1/2,-5/6,-7/6,-5/2,-7/2,-15/2,-21/2,-5/2,-25/6,
   -35/6,-3/2,-15/2,2,3,4,15,20,30,40,18,27,36,5,45,-5/2,-7/2,
   -9/2,-49/2,-63/2,-245/6,-105/2,-65/2,-91/2,-117/2,-35/6,
   -455/6,1,4/3,5/3,12,15,18,45/2,17,68/3,85/3,9/4,153/4;

ideal I=IP(n,G);

HF(n,I,0);
1
HF(n,I,1);
4
HF(n,I,2);
8
HF(n,I,3);
11
HF(n,I,4);
12
\end{verbatim}
Hence, the first difference of the Hilbert function fo this set of points is exactly $(1,3,4,3,1)$.
\end{Ex}

\end{document}